\documentclass[pdflatex,sn-mathphys-num]{sn-jnl}

\usepackage{graphicx}%
\usepackage{multirow}%
\usepackage{amsmath,amssymb,amsfonts}%
\usepackage{amsthm}%
\usepackage{mathrsfs}%
\usepackage{mathtools}
\usepackage[title]{appendix}%
\usepackage{xcolor}%
\usepackage{textcomp}%
\usepackage{manyfoot}%
\usepackage{booktabs}%
\usepackage{cleveref}
\usepackage{algorithm}%
\usepackage{algorithmicx}%
\usepackage{algpseudocode}%
\usepackage{listings}%
\numberwithin{equation}{section}
\theoremstyle{thmstyleone}%

\newtheorem{theorem}{Theorem}

\newtheorem{assumption}{Assumption}
\theoremstyle{remark}
\newtheorem{remark}{Remark}
\theoremstyle{definition}

\newtheorem{problem}{Problem}
\newcommand{\R}{\mathbb{R}}

\newcommand{\E}{\mathbb{E}}
\newcommand{\Prob}{\mathbb{P}}
\newcommand{\calL}{\mathcal{L}}
\newcommand{\calU}{\mathcal{U}}
\newcommand{\calT}{\mathcal{T}}
\newcommand{\calP}{\mathcal{P}}
\newcommand{\calE}{\mathcal{E}}
\newcommand{\calF}{\mathcal{F}}
\newcommand{\vstar}{v^{\star}}

\begin{document}

\title{Controlling the Swarm: Sparse Actuation and Collision Avoidance under Stochastic Delay
}

\author[1,2]{\fnm{Jiguang} \sur{Yu}}\email{jyu678@bu.edu}
\equalcont{This manuscript was accpeted for the 3rd annual Northeast Systems and Control Workshop (NESCW) at Princeton University. This version is for collecting suggestions and is not ready for peer review.}



\affil[1]{\orgdiv{Division of Systems Engineering},
  \orgname{Boston University},
  \orgaddress{\city{Boston}, \postcode{02215}, \state{MA}, \country{USA}}}

\affil[2]{\orgdiv{Department of Electrical and Computer Engineering},
  \orgname{Boston University},
  \orgaddress{\city{Boston}, \postcode{02215}, \state{MA}, \country{USA}}}


\abstract{
Classical flocking models demonstrate how local interactions generate emergent order, but real-world multi-agent deployments are bound by severe constraints: limited actuator availability, heterogeneous communication latencies, and environmental noise. In this talk, we present a unified finite-$N$ framework that tackles the interplay of these exact mechanisms. We study a delayed stochastic leader--follower particle system featuring topological communication, singular repulsion, and bounded sparse leader actuation.

A central challenge in such systems is mathematical well-posedness, as discontinuous communication laws and singular repulsions clash with standard strong It\^o frameworks. We resolve this by introducing an augmented Lyapunov functional that simultaneously enforces a strict collision barrier and closes a uniform Gr\"onwall estimate. Building on this rigorous foundation, we formulate a free-terminal-time, chance-constrained optimal control problem. We show that temporally sparse, bang--off--bang leader actuation not only drastically reduces control effort compared to continuous baselines, but also reveals non-monotone sensitivities to leader density. Ultimately, we demonstrate that in delayed stochastic swarms, adding more direct actuation is not strictly optimal---highlighting a highly non-trivial resource allocation paradox in cooperative control.
}

\keywords{Multi-agent systems, delayed stochastic dynamics, sparse control, leader--follower swarms, collision avoidance, chance-constrained optimization, augmented Lyapunov functional.}

\maketitle

\tableofcontents

\section{Introduction}
\label{sec:intro}
Coordinating large populations of autonomous agents under limited actuation,
communication delay, and stochastic forcing is a central challenge in modern
multi-agent systems. In realistic deployments, only a small subset of agents
can be directly actuated, information is received with heterogeneous
pair-dependent latencies, and environmental uncertainty perturbs the
dynamics through both idiosyncratic and common noise. These three
features---sparse leader intervention, delayed information flow, and
stochastic disturbance---are not peripheral complications; they are often the
dominant mechanisms shaping collective performance.

Classical flocking and consensus models provide the conceptual foundation for
emergent coordination, but they do not by themselves address this full
combination of effects. Deterministic Cucker--Smale-type models explain how
local interactions can generate large-scale order
\cite{cucker2007emergent,cucker2007mathematics,ha2009simple,motsch2011new,
motsch2014heterophilious}, while control-oriented leader--follower
extensions show that sparse actuation can influence global behavior
\cite{caponigro2013sparse,caponigro2015sparse,fornasier2014mean,
borzi2015modeling,albi2014boltzmann,bongini2017mean}. Delayed stochastic
variants, however, introduce two additional difficulties: the alignment term
depends on past information rather than the instantaneous state
\cite{erban2016cucker,choi2017cucker,pignotti2018convergence,liu2014flocking},
and singular repulsive interactions require a separate collision-avoidance
analysis before one can pass from local to global well-posedness
\cite{cucker2010avoiding,park2010cucker,carrillo2017sharp}. 

This paper develops a unified finite-$N$ framework that brings these threads
together. We study a delayed stochastic leader--follower particle system with
topological communication, singular repulsion, bounded sparse leader
actuation, and a free-terminal-time chance-constrained control objective. A
central modeling issue is that the physically natural communication law is
discontinuous in the joint configuration: topological ranks and availability
indicators jump when neighbor order changes. Since such discontinuities are
not compatible with the strong It\^o framework used for the well-posedness
theory, we explicitly distinguish between an exact discontinuous
communication law, used at the modeling level, and a regularized
surrogate, used in the analysis. We do not pursue the singular limit from
the regularized system back to the exact discontinuous one.

Our contributions are fourfold. First, we formulate a free-terminal-time
chance-constrained sparse-leader control problem in which the objective
balances terminal speed, actuation sparsity, and probabilistic safety.
Second, for the regularized delayed stochastic particle system, we prove
local strong well-posedness on the collision-free region. Third, under a
model-dependent barrier condition preventing collisions, we extend the local
solution globally, via an augmented Lyapunov functional that
simultaneously enforces a collision barrier and closes a Gr\"onwall-type
estimate on the state norm. Fourth, we present a computational case study
showing that a sparse feedback benchmark can outperform a continuously active
baseline at the nominal design point, while revealing a nontrivial feasible
window in terminal time and a non-monotone dependence on both delay and
leader density.

The paper is organized as follows. \Cref{sec:related} situates the paper
within the related literature. \Cref{sec:network} introduces the exact and
regularized communication laws. \Cref{sec:dynamics} formulates the delayed
stochastic particle dynamics and admissible controls. \Cref{sec:control}
presents the time-optimal sparse-leader problem. \Cref{sec:collision} and
\Cref{sec:collision-analysis} develop the local and global well-posedness
theory through a collision-avoidance barrier argument based on an augmented
Lyapunov functional. \Cref{sec:continuum} discusses the non-Markovian
continuum outlook induced by pair-dependent delays. Finally,
\Cref{sec:numerics} reports the numerical case study.
\section{Related Work}
\label{sec:related}

Our contributions sit at the intersection of four literatures: Cucker--Smale-type
alignment models, sparse leader--follower control, delayed and stochastic
flocking, and chance-constrained stochastic control. We briefly situate the
present paper within each, and then comment on the continuum (path-space)
viewpoint invoked later in \Cref{sec:continuum}.

\paragraph{Alignment models and their mean-field limits}

The Cucker--Smale model \cite{cucker2007emergent,cucker2007mathematics}
introduced a distance-dependent averaging law that has since become a
canonical vehicle for the analysis of emergent flocking. Unconditional and
conditional flocking theorems were sharpened in \cite{ha2009simple}, and the
role of communication-weight normalization was clarified by Motsch and Tadmor
\cite{motsch2011new,motsch2014heterophilious}, whose analysis of
heterophilious dynamics underlies the normalized alignment operator we use in
\eqref{eq:alignop}. Kinetic and hydrodynamic descriptions in the mean-field
limit were developed in \cite{ha2008particle,carrillo2010asymptotic}. These
classical Markovian continuum formulations provide the backdrop against which
the pair-dependent delays of our model generate genuinely non-Markovian
behavior (see \Cref{sec:continuum}).

\paragraph{Sparse leader--follower control}

The sparse-control viewpoint was pioneered by Caponigro, Fornasier, Piccoli,
and Tr\'elat \cite{caponigro2013sparse,caponigro2015sparse}, who showed that
component-wise $\ell^1$-in-space penalties on the control steer flocks to
consensus using actuation concentrated on a small leader subset. The
mean-field counterpart is developed in
\cite{fornasier2014mean,bongini2017mean}, and leader-driven refinements
appear in \cite{borzi2015modeling,albi2014boltzmann}. Our $L^1$-in-time
control cost and bang--off--bang sparsity heuristic (\Cref{rem:heuristic})
follow this framework, but we depart from it in three important ways: we
work in a stochastic setting with pair-dependent delays, we replace the
classical deterministic consensus target with a probabilistic target tube,
and we treat the horizon as a free terminal-time decision variable.

\paragraph{Delayed and stochastic alignment}

Delay effects in alignment dynamics have been studied in both deterministic
and stochastic forms. For constant or uniform delays, unconditional flocking
results were obtained in
\cite{liu2014flocking,pignotti2018convergence,choi2017cucker}, the latter two
incorporating normalized interaction weights and time-varying delays.
Noise-driven flocking, with or without a common component, is treated in
\cite{ha2009emergence,ahn2010stochastic,cucker2008flocking}. The combined
delay-plus-noise setting closest to ours is \cite{erban2016cucker}, which
studies a stochastic Cucker--Smale model with noise and delay but does not
address sparse actuation, singular repulsion, or chance constraints. Our
model subsumes these features: it has pair-dependent (rather than single
common) delays, both idiosyncratic and common noise, a singular pairwise
repulsion, and an explicit controlled leader subset.

\paragraph{Collision avoidance in singular alignment}

Deterministic collision avoidance in Cucker--Smale models with singular
pairwise repulsion was established by Cucker and Dong
\cite{cucker2010avoiding}, with sharp sufficient conditions later given by
Carrillo, Choi, Mucha, and Peszek \cite{carrillo2017sharp} and
bonding-force extensions in \cite{park2010cucker}. The
Lyapunov-functional argument in our \Cref{sec:collision-analysis} and
Appendix~\ref{app:collision-proof} is in the same spirit but adapted to the
stochastic setting: the augmented functional must additionally absorb the
diffusion trace term, and the singular-force cancellation must be combined
with a Gr\"onwall estimate that closes uniformly in the localization radius.
To our knowledge, the combination of pair-dependent delay, common noise,
singular repulsion, and sparse actuation has not previously been analyzed in
this form.

\paragraph{Chance-constrained stochastic control}

Chance constraints are a standard device for imposing probabilistic
feasibility on optimization problems with stochastic parameters
\cite{banerjea1998stochastic,nemirovski2007convex}. The use of sample-based and
scenario-based approximations in stochastic model predictive control is reviewed in
\cite{mesbah2016stochastic}. We use the sample-average formulation in
\Cref{rem:computational-chance}; the emphasis of the present paper is not
the approximation itself but its combination with a free terminal time, a
sparse $L^1$-in-time control cost, and a singular-interaction collision-safety
constraint inside the target tube.

\paragraph{Mean-field and path-space outlook}

The mean-field limit of controlled interacting particle systems is classical
in the Markovian case
\cite{fornasier2014mean,carmona2018probabilistic,bongini2017mean}, but
the pair-dependent delays in our model yield, as discussed in
\Cref{sec:continuum}, a genuinely non-Markovian path-space description. A
natural framework for such limits is Lacker's work on closed-loop equilibria
and path-space McKean--Vlasov systems \cite{lacker2020convergence}, and the
trajectory-level formulation developed in
\cite{carmona2018probabilistic,carrillo2010asymptotic,ha2008particle}. We do
not pursue the limit in this paper; we record the path-space viewpoint only
to clarify that any PDE-based closure in our setting requires assumptions
beyond those used in the finite-$N$ analysis.

\section{Communication Law: Exact vs.\ Regularized}
\label{sec:network}
This section isolates a distinction that is essential for the subsequent
analysis. The physically interpretable communication law is, in general,
discontinuous in the joint configuration and is therefore unsuitable for the
strong It\^o theory developed later. Accordingly, we distinguish between (i)
an exact discontinuous law, used at the modeling level, and (ii) a regularized surrogate, used in the well-posedness analysis. No
singular-limit result as the regularization parameter tends to zero is
claimed in the present paper.

We first introduce notations used in this paper. $|\cdot|$ denotes the
Euclidean norm. $\calP_2(\R^{k})$ is the set of Borel probability measures on
$\R^{k}$ with finite second moment. $\delta_z$ is the Dirac measure at $z$.
$C_b^1$ denotes the space of bounded continuously differentiable functions
with bounded derivative. For a random variable $X$ and an event $A$, we
write $\E[X]$ and $\Prob(A)$ for expectation and probability on the
underlying filtered probability space
$(\Omega,\calF,\mathbb{F},\Prob)$ satisfying the usual conditions.

Let $\mathbf{x}=(x_1,\dots,x_N)\in(\R^d)^N$ denote the joint position
configuration of the $N$ agents. At the modeling level, communication is
described by:
\begin{enumerate}[label=(\roman*)]
\item an exact availability indicator
$\chi_{ij}^{\mathrm{ex}}(t,\mathbf{x})\in\{0,1\}$ for $i\neq j$, encoding whether
agent $j$ can communicate with agent $i$ at time $t$;
\item an exact topological rank $r_{ij}^{\mathrm{ex}}(\mathbf{x})\in\{1,\dots,N-1\}$
for $i\neq j$, giving the distance-based rank of agent $j$ relative to
agent $i$.
\end{enumerate}
Both objects are generally discontinuous in $\mathbf{x}$: when two or more
agents are equidistant from agent $i$, the neighbor ordering may change
discontinuously, and any communication rule based on that ordering may jump
as well. For this reason, the
exact law is not used directly in the strong well-posedness theory of
\Cref{sec:collision}.

Fix a regularization parameter $\varepsilon>0$. For each $i\neq j$, introduce
regularized coefficients
\begin{equation}
\begin{aligned}
  \chi_{ij}^{\varepsilon} &\in C\big([0,\infty)\times(\R^d)^N;[0,1]\big), \\[4pt]
  r_{ij}^{\varepsilon} &\in C\big((\R^d)^N;[1,N-1]\big),
\end{aligned}
\end{equation}
which are understood as continuous surrogates of the exact indicator and
rank. We assume that, away from switching surfaces and rank-tie
configurations,
$\chi_{ij}^{\varepsilon}(t,\mathbf{x}) \to \chi_{ij}^{\mathrm{ex}}(t,\mathbf{x})$ and
$r_{ij}^{\varepsilon}(\mathbf{x}) \to r_{ij}^{\mathrm{ex}}(\mathbf{x})$ as
$\varepsilon\downarrow 0$.

Let $\phi\in C_b^1([0,\infty);[0,\infty))$ be a bounded communication
profile, and let $K\in\{1,\dots,N-1\}$ be the nominal topological interaction
number. For $i\neq j$, define the regularized communication weights by
\begin{equation}\label{eq:weight}
  a_{ij}^{\varepsilon}(t,\mathbf{x}) := \chi_{ij}^{\varepsilon}(t,\mathbf{x})\,
    \phi\Bigl(\frac{r_{ij}^{\varepsilon}(\mathbf{x})}{K}\Bigr).
\end{equation}
By convention, we set $a_{ii}^{\varepsilon}(t,\mathbf{x})\coloneqq 0$ for
$i=1,\dots,N$.
\begin{remark}\label{rem:weak-framework}
If one wishes to work directly with the exact discontinuous laws
$\chi_{ij}^{\mathrm{ex}}$ and $r_{ij}^{\mathrm{ex}}$, then a weak-solution or
martingale-problem framework is more natural than the strong It\^o
framework adopted below.
\end{remark}
\section{Particle Dynamics and Admissible Controls}
\label{sec:dynamics}

Fix a filtered probability space $(\Omega,\calF,\mathbb F,\Prob)$ satisfying
the usual conditions. Fix a maximal communication delay $\tau_{\max}>0$. For
each pair $(i,j)$, let $\tau_{ij}:[0,\infty)\to[0,\tau_{\max}]$ be a
measurable delay function. Let $W_1,\dots,W_N$ be independent $\R^{r}$-valued
Brownian motions and $W^0$ an additional $\R^{r_0}$-valued Brownian motion
representing a common noise.

For $t\ge 0$, define the empirical measure
$\mu_t^N := \frac{1}{N}\sum_{k=1}^N \delta_{(x_k(t),v_k(t))} \in \calP_2(\R^{2d})$
and the total regularized communication weight for each agent
$i=1,\dots,N$ as
\begin{equation}\label{eq:eta}
  \eta_i^\varepsilon(t,\mathbf{x}) := \sum_{j\neq i} a_{ij}^\varepsilon(t,\mathbf{x}).
\end{equation}
To avoid the division-by-zero singularity in the normalized alignment term
while preserving a strong-solution framework, we introduce an additional
normalization parameter $\delta>0$ and define
\begin{equation}\label{eq:alignop}
  \mathcal{A}_i^{\varepsilon,\delta}(t,\mathbf{x},\mathbf{v}_t)
  \coloneqq \frac{1}{\eta_i^\varepsilon(t,\mathbf{x})+\delta}
    \sum_{j\neq i} a_{ij}^\varepsilon(t,\mathbf{x})\,
    \bigl(v_j(t-\tau_{ij}(t))-v_i(t)\bigr),
\end{equation}
where $\mathbf{v}_t(\theta):=\mathbf{v}(t+\theta)$ for $\theta\in[-\tau_{\max},0]$ is
shorthand for the delayed velocity segment. In particular,
$\mathcal A_i^{\varepsilon,\delta}$ is defined for every configuration
$\mathbf{x}$, including those where $\eta_i^\varepsilon(t,\mathbf{x})=0$.

For $t\ge 0$, the particle dynamics are given by
\begin{equation}\label{eq:pos}
  dx_i(t)=v_i(t)\,dt
\end{equation}
and
\begin{equation}\label{eq:vel}
\begin{aligned}
  dv_i(t) &= \Biggl[ \mathcal A_i^{\varepsilon,\delta}\bigl(t,\mathbf{x}(t),\mathbf{v}_t\bigr)
    - \frac{1}{N}\sum_{j\neq i}\nabla U\bigl(x_i(t)-x_j(t)\bigr) \\
  &\qquad - \nabla V_{\mathrm{obs}}\bigl(x_i(t)\bigr)
    - b_i\bigl(v_i(t)-\vstar\bigr) + B_i u_i(t) \Biggr]dt \\
  &\quad +\sigma_i\bigl(x_i(t),v_i(t),\mu_t^N\bigr)\,dW_i(t) \\
  &\quad +\sigma_i^{0}\bigl(x_i(t),v_i(t),\mu_t^N\bigr)\,dW^0(t),
\end{aligned}
\end{equation}
where
$\sigma_i:\R^d\times\R^d\times\calP_2(\R^{2d})\to\R^{d\times r}$ and
$\sigma_i^0:\R^d\times\R^d\times\calP_2(\R^{2d})\to\R^{d\times r_0}$ are the
idiosyncratic and common diffusion coefficients. The drift terms have the
following interpretations:
\begin{enumerate}[label=(\roman*)]
\item Delayed alignment: normalized weighted averaging of delayed
neighbor velocities via $\mathcal A_i^{\varepsilon,\delta}$;
\item Pairwise interaction: $U=U_{\mathrm{rep}}+U_{\mathrm{form}}$,
consisting of singular repulsion and smooth attraction;
\item Obstacle potential $V_{\mathrm{obs}}$: environmental
confinement or obstacle avoidance;
\item Velocity pinning $-b_i(v_i-\vstar)$ and control input $B_i u_i$: active only on leaders.
\end{enumerate}

The initial data are prescribed on $[-\tau_{\max},0]$ by paths
$x_i^{\mathrm{in}},v_i^{\mathrm{in}} \in C([-\tau_{\max},0];\R^d)$ satisfying
the kinematic compatibility
\begin{equation}\label{eq:compat}
  x_i^{\mathrm{in}}(t) = x_i^{\mathrm{in}}(0) - \int_t^0 v_i^{\mathrm{in}}(s)\,ds,
  \quad t\in[-\tau_{\max},0].
\end{equation}
We impose $x_i(t)=x_i^{\mathrm{in}}(t)$ and $v_i(t)=v_i^{\mathrm{in}}(t)$
for all $t\in[-\tau_{\max},0]$.

We assume $U(z)=U_{\mathrm{rep}}(z)+U_{\mathrm{form}}(z)$ for $z\neq 0$. The
first collision time is defined by
$\tau_{\mathrm{coll}} := \inf\{ t\ge 0:\min_{i\neq j}|x_i(t)-x_j(t)|=0 \}$
with $\inf\varnothing=+\infty$. In the analysis, local well-posedness is
first established on the collision-free region
$\{(\mathbf{x},\mathbf{v}):\ x_i\neq x_j \text{ for all } i\neq j\}$, and global
existence follows if $\tau_{\mathrm{coll}}=+\infty$ almost surely
\cite{cucker2010avoiding,carrillo2017sharp,park2010cucker}.

Let $\calL\subset\{1,\dots,N\}$ be the leader set with $|\calL|$ small
compared to $N$. For followers $i\notin\calL$, we set $b_i, B_i, u_i = 0$.
For a horizon $T>0$, the admissible controls are deterministic open-loop
controls:
\begin{equation}\label{eq:admissible}
\begin{aligned}
  \calU_T \coloneqq \bigl\{ &u=(u_i)_{i\in\calL} :
    u_i \in L^\infty(0,T;\R^{m_i}), \\[4pt]
  &\quad\quad \|u_i\|_{L^\infty}\le M_i \text{ for } i\in\calL \bigr\},
\end{aligned}
\end{equation}
where $M_i>0$ are prescribed actuator bounds.
\section{Time-Optimal Sparse Leader Control}
\label{sec:control}
\subsection{Performance metrics and target tube}
\label{subsec:metrics}
We now formulate the finite-$N$ sparse-leader control problem associated
with the delayed stochastic particle system of \Cref{sec:dynamics}; the
sparse-actuation perspective follows the line of work initiated by
\cite{caponigro2013sparse,caponigro2015sparse} and continued in
\cite{fornasier2014mean,albi2014boltzmann,bongini2017mean}.
Let $d_{ij}^{\star}\ge 0$ ($i,j=1,\dots,N$) denote the prescribed pairwise
formation distances, and let $\Psi:(0,\infty)\to[0,\infty)$ be a nonnegative
barrier-type penalty, for example $\Psi(r)=r^{-p}$ with $p>0$. For
$\mathbf{x}=(x_1,\dots,x_N)\in(\R^d)^N$ and $\mathbf{v}=(v_1,\dots,v_N)\in(\R^d)^N$,
define the performance metrics
\begin{subequations}
\begin{align}
  \calE_{\mathrm{vel}}(\mathbf{v}) &:= \frac{1}{N}\sum_{i=1}^{N}|v_i-\vstar|^2, \label{eq:Evel}\\
  \calE_{\mathrm{form}}(\mathbf{x}) &:= \frac{1}{N^2}\sum_{i\neq j}
    \Bigl||x_i-x_j|-d_{ij}^{\star}\Bigr|^2, \label{eq:Eform}\\
  \calE_{\mathrm{safe}}(\mathbf{x}) &:= \frac{1}{N^2}\sum_{i\neq j}
    \Psi\bigl(|x_i-x_j|\bigr). \label{eq:Esafe}
\end{align}
\end{subequations}
Here, $\calE_{\mathrm{vel}}$ measures velocity mismatch from the reference
$\vstar$, $\calE_{\mathrm{form}}$ measures deviation from the target
formation geometry, and $\calE_{\mathrm{safe}}$ penalizes small inter-agent
distances.

For tolerances $\varepsilon_v,\delta_f,\rho>0$---representing an admissible
terminal velocity error, formation error, and minimum separation,
respectively---we define the target tube
\begin{equation}\label{eq:target}
\begin{aligned}
  \calT_{\varepsilon_v,\delta_f,\rho}
  \coloneqq \Bigl\{ &(\mathbf{x},\mathbf{v})\in(\R^d)^{2N} :
    \calE_{\mathrm{vel}}(\mathbf{v})\le\varepsilon_v^{\,2}, \\
  &\,\,\,\,\,\calE_{\mathrm{form}}(\mathbf{x})\le\delta_f^{\,2},\
    \min_{i\neq j}|x_i-x_j|\ge\rho \Bigr\}.
\end{aligned}
\end{equation}

\begin{remark}[Separation of regularization and tolerance parameters]
\label{rem:eps-delta-disambiguation}
The target-tube tolerances $\varepsilon_v,\delta_f$ are design
quantities specifying admissible terminal errors; they are logically
independent of the regularization parameters $\varepsilon,\delta$ introduced
in \eqref{eq:weight}--\eqref{eq:alignop}, which merely smooth out
discontinuities in the communication law and the division-by-zero in the
normalized alignment operator. The subscripts $v$ (velocity) and $f$
(formation) are used throughout the paper to avoid notational ambiguity.
\end{remark}

\subsection{Chance-constrained free-terminal-time formulation}
\label{subsec:chance}
The central engineering question is: what is the smallest time in which the
leader set can steer the swarm into the target tube with a prescribed
success probability, while using sparse bounded actuation? Fix a risk
tolerance $\alpha\in(0,1)$ and weights $\lambda_1,\dots,\lambda_4\ge 0$.
Chance-constrained formulations of this type are standard in stochastic
optimization and stochastic model predictive control
\cite{nemirovski2007convex,campi2008exact,mesbah2016stochastic,
banerjea1998stochastic}.
\begin{problem}[Time-optimal sparse leader control]\label{prob:main}
Find a terminal time $T>0$ and an admissible control $u\in\calU_T$ that
minimize $J(T,u)$ subject to the terminal chance constraint
\begin{equation}\label{eq:chance}
  \Prob\Bigl( \bigl(\mathbf{x}(T),\mathbf{v}(T)\bigr)\in\calT_{\varepsilon_v,\delta_f,\rho} \Bigr)
    \ge 1-\alpha,
\end{equation}
where the cost functional is
\begin{equation}\label{eq:J}
\begin{aligned}
  J(T,u) \coloneqq\;& T + \lambda_1\sum_{i\in\calL}\int_0^T |u_i|_2\,dt \\
  &+ \E\biggl[ \int_0^T (\lambda_2\calE_{\mathrm{vel}}
     + \lambda_3\calE_{\mathrm{safe}})\,dt
     + \lambda_4\calE_{\mathrm{form}}(\mathbf{x}(T)) \biggr].
\end{aligned}
\end{equation}
\end{problem}
Three modeling choices are worth emphasizing:
\begin{enumerate}[label=(\roman*)]
\item Free terminal time.\enspace The horizon $T$ is a decision
variable, balancing convergence speed against performance costs.
\item $L^1$-in-time control penalty.\enspace The term
$\lambda_1\sum_{i\in\calL}\int_0^T |u_i(t)|_2\,dt$ promotes temporal
sparsity: the optimizer is incentivized to switch actuation off when passive
dynamics are favorable.
\item Hard terminal reliability with soft terminal shaping.\enspace
The chance constraint \eqref{eq:chance} imposes a hard probabilistic
requirement, while $\lambda_4\calE_{\mathrm{form}}(\mathbf{x}(T))$ favors terminal
states concentrated near the desired shape.
\end{enumerate}
\begin{remark}[Chance vs.\ almost-sure constraints]\label{rem:as-infeasible}
A stronger requirement
$\Prob( (\mathbf{x}(T),\mathbf{v}(T))\in\calT_{\varepsilon_v,\delta_f,\rho} )=1$ is
typically infeasible under nondegenerate Brownian perturbations, as
diffusion spreads probability mass over the state space
\cite{banerjea1998stochastic,nemirovski2007convex}.
\end{remark}
\begin{remark}[Stopping-time alternative]\label{rem:stopping}
An alternative replaces $T$ with the first hitting time
$\tau_{\calT} := \inf\{ t\ge 0: (\mathbf{x}(t),\mathbf{v}(t))\in\calT_{\varepsilon_v,\delta_f,\rho} \}$,
minimizing
\begin{equation*}
\begin{aligned}
  J_{\mathrm{hit}}(u) \coloneqq\;&
  \E\biggl[ \tau_{\calT}
    + \lambda_1\sum_{i\in\calL}\int_0^{\tau_{\calT}} |u_i|_2\,dt \\
  &\quad + \int_0^{\tau_{\calT}} (\lambda_2\calE_{\mathrm{vel}}
    + \lambda_3\calE_{\mathrm{safe}})\,dt \biggr].
\end{aligned}
\end{equation*}
\end{remark}
\begin{remark}[Computational treatment]\label{rem:computational-chance}
The constraint \eqref{eq:chance} is enforced via sample-based approximation:
simulating $M$ independent realizations and estimating the probability by
the empirical frequency.
\end{remark}
\begin{remark}[Formal sparsity heuristic]\label{rem:heuristic}
The $L^1$ cost suggests bang--off--bang behavior. In a Hamiltonian
description,
$H \approx \sum_{i\in\calL} [ p_i^\top B_i u_i + \lambda_1 |u_i|_2 ]$,
indicating that the optimal control should vanish when the switching signal
is below a threshold, as in the sparse-stabilization framework of
\cite{caponigro2013sparse,caponigro2015sparse,bongini2017mean}.
\end{remark}

\section{Collision Avoidance and Well-Posedness}
\label{sec:collision}

We now collect the assumptions under which the regularized delayed
stochastic particle system admits a pathwise unique strong solution up to
the first collision time, and globally provided collisions are excluded.
\begin{assumption}[Regularized communication law]\label{ass:net}
The communication profile satisfies
$\phi\in C_b^1([0,\infty);[0,\infty))$. For each $i\neq j$, the regularized
coefficients $\chi_{ij}^{\varepsilon}:[0,\infty)\times(\R^d)^N\to[0,1]$ and
$r_{ij}^{\varepsilon}:(\R^d)^N\to[1,N-1]$ are bounded and locally Lipschitz
in $\mathbf{x}$, uniformly on compact time intervals. Consequently, the weights
$a_{ij}^{\varepsilon}(t,\mathbf{x}) = \chi_{ij}^{\varepsilon}(t,\mathbf{x})\,
\phi(r_{ij}^{\varepsilon}(\mathbf{x})/K)$ are also bounded and locally Lipschitz
in $\mathbf{x}$.
\end{assumption}
\begin{assumption}[Delay measurability]\label{ass:delay}
For all $i\neq j$, the delay functions
$\tau_{ij}:[0,\infty)\to[0,\tau_{\max}]$ are measurable.
\end{assumption}
\begin{assumption}[Regularized alignment map]\label{ass:align}
Recall $\eta_i^\varepsilon$ and $\mathcal A_i^{\varepsilon,\delta}$ from
\eqref{eq:eta}--\eqref{eq:alignop}. We assume that for each $i$, on every
set $[0,T]\times\mathcal K\times\mathcal H$ where
$\mathcal K\subset(\R^d)^N$ is compact and collision-free and
$\mathcal H\subset C([-\tau_{\max},0];(\R^d)^N)$ is bounded (in the
supremum norm), the map $(t,\mathbf{x},\mathbf{v}_t)\mapsto
\mathcal A_i^{\varepsilon,\delta}(t,\mathbf{x},\mathbf{v}_t)$ is locally Lipschitz in
$(\mathbf{x},\mathbf{v}_t)$, uniformly in $t\in[0,T]$. Moreover,
$\mathcal A_i^{\varepsilon,\delta}$ has at most linear growth in
$(\mathbf{x},\mathbf{v}_t)$ (which follows from \Cref{ass:net} and the denominator
regularization by $\delta>0$).
\end{assumption}
\begin{assumption}[Potential regularity]\label{ass:pot}
The potentials $V_{\mathrm{obs}}, U_{\mathrm{form}} \in C^1(\R^d)$ have
locally Lipschitz gradients with at most linear growth. The repulsive
potential $U_{\mathrm{rep}}\in C^2(\R^d\setminus\{0\})$ is radial and
singular at the origin: $U_{\mathrm{rep}}(z)\to+\infty$ as $|z|\downarrow 0$.
Accordingly, $U = U_{\mathrm{rep}} + U_{\mathrm{form}}$ has a locally
Lipschitz gradient on every collision-free subset of $(\R^d)^N$.
\end{assumption}
\begin{assumption}[Noise regularity]\label{ass:noise}
For each $i$, the diffusion coefficients
$\sigma_i:\R^d\times\R^d\times\calP_2(\R^{2d})\to\R^{d\times r}$ and
$\sigma_i^0:\R^d\times\R^d\times\calP_2(\R^{2d})\to\R^{d\times r_0}$ are
locally Lipschitz in $(x,v,\mu)$ (where $\calP_2$ is equipped with the
$W_2$ metric) and satisfy linear growth bounds.
\end{assumption}
\begin{assumption}[Collision-free compatible history]\label{ass:hist}
The history paths
$x_i^{\mathrm{in}},v_i^{\mathrm{in}} \in C([-\tau_{\max},0];\R^d)$ satisfy
\eqref{eq:compat} and are collision-free:
$\inf_{s\in[-\tau_{\max},0]}\min_{i\neq j}
|x_i^{\mathrm{in}}(s)-x_j^{\mathrm{in}}(s)|>0$.
\end{assumption}

Let $\mathsf{Conf}_N(\R^d) := \{ \mathbf{x}\in(\R^d)^N : x_i\neq x_j \text{ for all } i\neq j \}$
be the collision-free configuration space, and recall the first collision
time
$\tau_{\mathrm{coll}} := \inf\{ t\ge 0 : \mathbf{x}(t)\notin\mathsf{Conf}_N(\R^d) \}$
(with $\inf\varnothing=+\infty$).
\begin{theorem}[Local strong well-posedness]\label{thm:local}
Suppose Assumptions~\ref{ass:net}--\ref{ass:hist} hold. For any $T>0$ and
$u\in\calU_T$, the system \eqref{eq:pos}--\eqref{eq:vel} admits a pathwise
unique strong solution on $[-\tau_{\max},\,T\wedge\tau_{\mathrm{coll}})$.
\end{theorem}
\begin{proof}[Proof sketch]
On $\mathsf{Conf}_N(\R^d)$, the singular drift $\nabla U_{\mathrm{rep}}$ is
locally Lipschitz. By Assumptions~\ref{ass:net}--\ref{ass:align}, the
regularized weights and alignment operator are also locally Lipschitz on
collision-free sets. Assumption~\ref{ass:noise} ensures Lipschitz diffusion.
Rewriting the system as a stochastic functional differential equation on
$C([-\tau_{\max},0];\R^{2dN})$, standard SFDE theory
\cite{mohammed1984stochastic,mao2007stochastic,von2010existence} up to the first
exit time from bounded collision-free sets applies. Details are in Appendix~\ref{app:local-proof}.
\end{proof}

\begin{theorem}[Global strong well-posedness]\label{thm:global}
Assume the hypotheses of \Cref{thm:local} and that the repulsive interaction
ensures $\Prob(\tau_{\mathrm{coll}}=+\infty)=1$. Then, for every $T>0$, the
solution extends uniquely to a pathwise unique strong solution on
$[-\tau_{\max},T]$.
\end{theorem}
\begin{proof}[Proof sketch]
If $\Prob(\tau_{\mathrm{coll}}=+\infty)=1$, then $\tau_{\mathrm{coll}}>T$
almost surely for any $T>0$. Thus, the local solution never reaches the
singular set on $[0,T]$, allowing the unique extension throughout the
interval.
\end{proof}
\begin{remark}[Role of the no-collision hypothesis]\label{rem:nocollision-role}
The global result is conditional on the no-collision property
$\Prob(\tau_{\mathrm{coll}}=+\infty)=1$. This property is typically verified
by Lyapunov or barrier arguments exploiting the $U_{\mathrm{rep}}$
singularity, which we develop in \Cref{sec:collision-analysis} and
rigorously verify in \Cref{app:collision-proof}.
\end{remark}
\section{Collision-Avoidance Analysis}
\label{sec:collision-analysis}

This section records the structure of the argument used to upgrade the local
strong solution of \Cref{thm:local} to the global strong solution of
\Cref{thm:global}. The full rigorous proof is given in
\Cref{app:collision-proof}; the goal here is to record the
Lyapunov-functional structure and to explain why the pure repulsion energy
is insufficient and why an augmented functional is required. The
scheme parallels the deterministic Cucker--Smale collision-avoidance
literature \cite{cucker2010avoiding,park2010cucker,carrillo2017sharp},
adapted to the stochastic setting by absorbing the diffusion trace into the
drift estimate \cite{ha2009emergence,ahn2010stochastic}.

\subsection{Augmented Lyapunov functional}
\label{subsec:augmented-lyapunov}

Recall
$\tau_{\mathrm{coll}} := \inf\{t\ge 0 : \min_{i\neq j}|x_i(t)-x_j(t)|=0\}$
with $\inf\varnothing = +\infty$. For $t<\tau_{\mathrm{coll}}$, define the augmented Lyapunov functional
\begin{equation}\label{eq:Htilde-body}
\begin{aligned}
  \widetilde{\mathscr H}(t)
  := \;&\frac{1}{2}\sum_{i=1}^{N}|x_i(t)|^2
       + \frac{1}{2}\sum_{i=1}^{N}|v_i(t)|^2 \\
       &+ \frac{1}{N}\sum_{1\le i<j\le N} U_{\mathrm{rep}}\bigl(x_i(t)-x_j(t)\bigr).
\end{aligned}
\end{equation}
Without loss of generality, we assume $U_{\mathrm{rep}}\ge 0$ throughout
this section (achievable by an additive constant when $U_{\mathrm{rep}}$ is
bounded below). On $[0,\tau_{\mathrm{coll}})$, $\widetilde{\mathscr H}$
enjoys two complementary properties:

\begin{enumerate}[label=(\roman*)]
\item Collision barrier. Since $U_{\mathrm{rep}}(z)\to+\infty$ as
$|z|\downarrow 0$ by \Cref{ass:pot}, if
$\min_{i\neq j}|x_i(t)-x_j(t)|\downarrow 0$ then
$\widetilde{\mathscr H}(t)\to+\infty$.
\item Coercivity in $(\mathbf{x},\mathbf{v})$. The spatial and kinetic moments
are controlled by $\widetilde{\mathscr H}$:
$\tfrac{1}{2}\bigl(|\mathbf{x}(t)|^2+|\mathbf{v}(t)|^2\bigr) \le \widetilde{\mathscr H}(t)$.
\end{enumerate}

\begin{remark}[Why the repulsion energy alone is insufficient]
\label{rem:why-augmented}
A natural first attempt is to base the barrier argument on the pure
repulsion energy $\Phi(t) := \sum_{i<j}U_{\mathrm{rep}}(x_i-x_j)$. However,
$\Phi$ has no coercivity in $(\mathbf{x},\mathbf{v})$, and its It\^o drift is
$\sum_{i<j}\nabla U_{\mathrm{rep}}(x_i-x_j)\cdot(v_i-v_j)$, which cannot be
bounded by $C(1+\Phi)$ since it depends on the velocity increments
$v_i-v_j$ that $\Phi$ does not control. Including
$\tfrac12\sum_i|v_i|^2$ is needed so that the singular drift cancels
against the chain-rule term on the repulsion energy
(\Cref{subsec:augmented-ito}). Including $\tfrac12\sum_i|x_i|^2$ is then
needed so that the resulting Gr\"onwall estimate
(\Cref{subsec:augmented-drift-bound}) is uniform in the localization radius;
without that term, the smooth forces $\nabla U_{\mathrm{form}},
\nabla V_{\mathrm{obs}}$ and the diffusion trace cannot be dominated by the
functional alone. The augmented $\widetilde{\mathscr H}$ is therefore the
minimal object that simultaneously serves as a barrier and closes the
Gr\"onwall loop.
\end{remark}

\subsection{Localized It\^o identity}
\label{subsec:augmented-ito}

Fix $T>0$ and introduce the localization sequence
\begin{equation}\label{eq:tauR-body}
  \tau_R := \inf\{t\in[0,T\wedge\tau_{\mathrm{coll}}) :
    \widetilde{\mathscr H}(t)\ge R\}, \qquad R\ge 1.
\end{equation}
On $[0,\tau_R]$, the configuration remains collision-free and, by
coercivity, $|\mathbf{x}(t)|^2+|\mathbf{v}(t)|^2 \le 2R$ throughout. All coefficients of
\eqref{eq:pos}--\eqref{eq:vel} are therefore bounded on $[0,\tau_R]$, and
It\^o's formula may be applied term by term to \eqref{eq:Htilde-body}.

The crucial structural fact is that the contribution of the singular drift
$-\frac{1}{N}\sum_{j\neq i}\nabla U_{\mathrm{rep}}(x_i-x_j)$ in
$v_i\cdot dv_i$ cancels exactly with the chain-rule derivative of the
pairwise repulsion energy. Using the oddness of $\nabla U_{\mathrm{rep}}$
(from radiality) and the matching factor $\frac{1}{N}$ between the drift
and the energy,
\begin{equation}\label{eq:singular-cancel}
\begin{aligned}
  &-\sum_{i=1}^{N} v_i\cdot\frac{1}{N}\sum_{j\neq i}\nabla U_{\mathrm{rep}}(x_i-x_j) \\
  &\quad+ \frac{1}{N}\sum_{1\le i<j\le N}\nabla U_{\mathrm{rep}}(x_i-x_j)\cdot(v_i-v_j) = 0.
\end{aligned}
\end{equation}
Consequently, for $t\le T$,
\begin{equation}\label{eq:Htilde-ito-body}
  \widetilde{\mathscr H}(t\wedge\tau_R)
  = \widetilde{\mathscr H}(0)
  + \int_0^{t\wedge\tau_R}\widetilde{\Gamma}(s)\,ds
  + \widetilde{M}_R(t),
\end{equation}
where $\widetilde{M}_R$ is a local martingale and $\widetilde{\Gamma}$
contains only the smooth contributions---the delayed alignment, smooth
formation and obstacle forces, pinning and control terms, and the diffusion
trace---but no singular repulsive contribution. The explicit form of
$\widetilde{\Gamma}$ is given by \eqref{eq:appB-drift}.

\subsection{Drift estimate and barrier bound}
\label{subsec:augmented-drift-bound}

The decomposition \eqref{eq:Htilde-ito-body} reduces collision avoidance to
controlling the smooth drift $\widetilde{\Gamma}$. The key hypothesis,
stated precisely as \Cref{ass:appB-drift}, is that for every $T>0$ there is
a deterministic constant $C_T>0$, independent of $R$, such that
\begin{equation}\label{eq:drift-augmented-body}
  \widetilde{\Gamma}(t) \le C_T\bigl(1+\widetilde{\mathscr H}(t)\bigr),
  \qquad t\le T\wedge\tau_R,\ \text{a.s.}
\end{equation}
Because the singular potential has already been eliminated by
\eqref{eq:singular-cancel}, the bound \eqref{eq:drift-augmented-body} is a
standard consequence of the global linear-growth assumptions on the smooth
coefficients $\nabla U_{\mathrm{form}},\nabla V_{\mathrm{obs}},\sigma_i,
\sigma_i^0$, together with the linear growth of the regularized alignment
operator $\mathcal A_i^{\varepsilon,\delta}$ recorded in \Cref{ass:align}.
See \Cref{app:collision-proof}, Section~B.6, and
\cite{mao2007stochastic,khasminskii2011stochastic}.

Under \eqref{eq:drift-augmented-body}, the stopped local martingale
$\widetilde{M}_R$ is a true martingale on $[0,T]$, because its integrand is
bounded on $[0,\tau_R]$. Taking expectations in \eqref{eq:Htilde-ito-body}
yields
\begin{equation}\label{eq:gronwall-input}
\begin{aligned}
  \E[\widetilde{\mathscr H}(t\wedge\tau_R)]
  \le\;& \widetilde{\mathscr H}(0)
        + C_T\int_0^{t}\bigl(1+\E[\widetilde{\mathscr H}(s\wedge\tau_R)]\bigr)\,ds,
\end{aligned}
\end{equation}
and Gr\"onwall's lemma implies
\begin{equation}\label{eq:Htilde-bound}
  \sup_{R\ge 1}\sup_{0\le t\le T}
  \E[\widetilde{\mathscr H}(t\wedge\tau_R)] < \infty.
\end{equation}

\subsection{From the barrier estimate to no collision}
\label{subsec:augmented-no-collision}

Pathwise, $\tau_R\uparrow T\wedge\tau_{\mathrm{coll}}$ as $R\to\infty$.
Indeed, on $\{T<\tau_{\mathrm{coll}}\}$ the functional
$\widetilde{\mathscr H}$ is continuous and finite on $[0,T]$, so $\tau_R=T$
for all sufficiently large $R$; on $\{\tau_{\mathrm{coll}}\le T\}$ the
barrier property of \Cref{subsec:augmented-lyapunov} forces
$\widetilde{\mathscr H}(t)\to+\infty$ as $t\uparrow\tau_{\mathrm{coll}}$,
so $\tau_R\uparrow\tau_{\mathrm{coll}}$. By Fatou's lemma and
\eqref{eq:Htilde-bound},
\begin{equation}\label{eq:Htilde-fatou}
  \E\Bigl[\liminf_{R\to\infty}\widetilde{\mathscr H}(t\wedge\tau_R)\Bigr] <\infty,
  \qquad 0\le t\le T.
\end{equation}
If $\Prob(\tau_{\mathrm{coll}}\le T)>0$, then on that event
$\widetilde{\mathscr H}(t\wedge\tau_R)\to+\infty$ as $R\to\infty$,
contradicting \eqref{eq:Htilde-fatou}. Hence
\begin{equation}\label{eq:no-collision}
  \Prob\bigl(\tau_{\mathrm{coll}} = +\infty\bigr) = 1,
\end{equation}
which is the no-collision hypothesis required by \Cref{thm:global}.

\subsection{Model-dependent sufficient conditions}
\label{subsec:model-dependent}

Making the scheme rigorous in a concrete model amounts to verifying
\eqref{eq:drift-augmented-body} for the specific choices of
$U_{\mathrm{form}},V_{\mathrm{obs}},\sigma_i,\sigma_i^0$. Because the
singular $U_{\mathrm{rep}}$ does not appear in $\widetilde{\Gamma}$, the
drift bound reduces to a standard linear-growth estimate on the smooth
coefficients. \Cref{app:collision-proof}, Section~B.6, gives the
verification; see also
\cite{cucker2010avoiding,carrillo2017sharp,ha2009emergence,ahn2010stochastic}
for related singular and stochastic settings.

\begin{remark}[Role of the cancellation]
\label{rem:cancellation}
The cancellation \eqref{eq:singular-cancel} is the analytical core of the
argument: after cancellation, the drift $\widetilde\Gamma$ contains only
smooth contributions, and no pairwise term of the form
$|x_i-x_j|^{-p}$ appears. Together with the coercivity of
$\widetilde{\mathscr H}$, this is what allows a uniform-in-$R$ Gr\"onwall
estimate. The pure repulsion energy provides the barrier property but not
the coercivity; the pure kinetic energy provides the pairing with the
singular drift but no collision control. Only the augmented
$\widetilde{\mathscr H}$ provides both simultaneously.
\end{remark}

\begin{remark}[Two-step structure]
\label{rem:two-step-redux}
The strategy is two-step: first, prove local well-posedness on the
collision-free region (\Cref{thm:local}); second, establish the barrier
estimate \eqref{eq:no-collision} via the augmented Lyapunov functional
(this section and \Cref{app:collision-proof}).
\end{remark}
\section{Delay-Robust Communication and Continuum Outlook}
\label{sec:continuum}
The delay $\tau_{ij}(t)\in[0,\tau_{\max}]$ models the end-to-end latency
between the measurement of agent $j$'s state and its availability to
agent $i$'s controller at time $t$. Such latency may arise from wireless
channel access, sensing and estimation pipelines, packet transmission, or
geographical propagation effects. Because the delays are pair-dependent,
the alignment term at time $t$ depends explicitly on the collection of
delayed velocities $\{v_j(t-\tau_{ij}(t)):\ j\neq i\}$, and therefore on
the velocity history over the entire interval $[t-\tau_{\max},t]$, rather
than on a single common retarded time.

Accordingly, the delayed particle system is naturally non-Markovian in the
original state variables $(\mathbf{x}(t),\mathbf{v}(t))$. A Markovian reformulation is
possible only after enlarging the state space to include a suitable
history segment or memory variable.

In the formal mean-field limit $N\to\infty$, this memory effect prevents,
in general, a direct closure in terms of a classical Markovian
Vlasov--Fokker--Planck equation posed only on the instantaneous phase-space
variables $(x,v)\in\R^{2d}$
\cite{ha2008particle,carrillo2010asymptotic}. The natural continuum object
is instead a law on trajectory space, for example
\begin{equation}\label{eq:path-measure}
  \mu \in \calP\bigl(C([-\tau_{\max},T];\R^{2d})\bigr),
\end{equation}
or, equivalently, a time-indexed family of laws of path segments on
$\calP\bigl(C([-\tau_{\max},0];\R^{2d})\bigr)$. At this level, the limiting
dynamics should be interpreted as a path-dependent McKean--Vlasov-type
evolution \cite{carmona2018probabilistic,lacker2020convergence} rather than
as a closed Markovian kinetic equation on $\R^{2d}$.

For this reason, we do not postulate at the modeling level a
classical PDE closure involving only an instantaneous law or a single
delayed marginal such as $\mu_{t-\tau}$. Any reduction to a
finite-dimensional Markovian continuum description must be justified
separately and requires additional structure, such as a common constant
delay, an augmented-state representation, or a controlled
finite-dimensional memory truncation.

The present paper does not pursue the mean-field limit. We record the
path-space viewpoint only to clarify that pair-dependent communication
delays lead, in general, to a genuinely non-Markovian continuum
description, and that any PDE-based closure requires assumptions beyond
those used in the finite-$N$ analysis above.
\section{Numerical Case Study}
\label{sec:numerics}

The numerical experiments in this section are not intended to re-establish
the analytical well-posedness theory of
\Cref{sec:collision,sec:collision-analysis}. Their role is instead to
illustrate the control-design consequences of the model and to quantify the
trade-off between reliability, sparsity, and actuation effort. The study is
organized around three questions:
\begin{enumerate}[label=(\roman*),nosep]
\item can sparse leader actuation match or outperform a continuously active
baseline in satisfying the terminal chance constraint;
\item is there a nontrivial feasible window of terminal times;
\item how sensitive is terminal performance to leader density $|\calL|$ and
the communication-delay bound $\tau_{\max}$?
\end{enumerate}
This perspective is consistent with simulation-based assessment in
stochastic control \cite{mesbah2016stochastic} and with sample-based
evaluation of chance constraints \cite{nemirovski2007convex}.

\subsection{Experimental setup}
\label{subsec:num-setup}

We simulate $N=100$ agents evolving in $\R^2$. Unless otherwise stated, the
leader set has size $|\calL|=8$, and the maximum pair-dependent delay bound
is $\tau_{\max}=0.25\,\mathrm{s}$ at the nominal design point. The dynamics
include both idiosyncratic and common Brownian forcing. For each
realization, the pair-dependent delays $\tau_{ij}$ are drawn independently
and uniformly from $[0,\tau_{\max}]$ and then kept fixed over the
simulation horizon.

The target tube $\calT_{\varepsilon_v,\delta_f,\rho}$ from \eqref{eq:target}
is instantiated with
\begin{equation}\label{eq:tolerances-num}
  \varepsilon_v = 0.6, \qquad
  \delta_f = 3.0, \qquad
  \rho = 0.03,
\end{equation}
so that tube membership requires
\[
  \calE_{\mathrm{vel}}(T)\le \varepsilon_v^{\,2}=0.36,\qquad
  \calE_{\mathrm{form}}(T)\le \delta_f^{\,2}=9.0,\qquad
  \min_{i\neq j}|x_i(T)-x_j(T)|\ge 0.03.
\]
The prescribed terminal reliability is $1-\alpha=0.95$.

Time discretization is performed by an Euler--Maruyama scheme with step
$\Delta t=0.01\,\mathrm{s}$ \cite{higham2001algorithmic}. Delayed velocity values are retrieved from a
ring buffer containing
$\lceil\tau_{\max}/\Delta t\rceil+1$ snapshots. For each
parameter configuration, performance is evaluated over $M$ independent
sample paths, with $M=60$ for the nominal controller comparison and
$M=30$ for the sensitivity sweeps. All empirical success probabilities are
reported with 95\% Wilson score confidence intervals
\cite{brown2001interval}, which are preferable to Wald intervals in the
present regime of modest sample size and proportions near one.

We compare two bounded leader-feedback benchmarks built from the same raw
signal $u_i^{\mathrm{raw}}(t)$: a proportional--derivative response to the
leader's position error relative to the swarm centroid and velocity error
relative to $\vstar$, clipped componentwise to the actuator bound $M_i$.
\begin{enumerate}[label=(\roman*)]
\item \textbf{Baseline feedback.} The raw signal is applied continuously:
\[
  u_i(t)=u_i^{\mathrm{raw}}(t), \qquad i\in\calL,\ t\in[0,T].
\]
\item \textbf{Sparse feedback.} The same raw signal is thresholded:
\[
  u_i(t)=u_i^{\mathrm{raw}}(t)\,\mathbf{1}\{|u_i^{\mathrm{raw}}(t)|_2>\theta\},
\]
for a fixed threshold $\theta>0$, producing bang--off--bang actuation.
\end{enumerate}
This comparison isolates the effect of temporal sparsification: both
controllers use the same underlying feedback architecture and differ only
in whether weak control signals are suppressed. The sparse controller is
therefore a structured benchmark rather than the optimal solution of
\Cref{prob:main}.

For each realization we record the terminal metrics
$\calE_{\mathrm{vel}}(T)$, $\calE_{\mathrm{form}}(T)$, and
$\min_{i\neq j}|x_i(T)-x_j(T)|$, together with the component indicators
\[
  \mathbf{1}\{\calE_{\mathrm{vel}}(T)\le \varepsilon_v^{\,2}\},\qquad
  \mathbf{1}\{\calE_{\mathrm{form}}(T)\le \delta_f^{\,2}\},\qquad
  \mathbf{1}\{\min_{i\neq j}|x_i(T)-x_j(T)|\ge \rho\}.
\]
Their conjunction is the terminal tube indicator. We also record the
running actuation quantity
$\sum_{i\in\calL}\int_0^T |u_i(t)|_2\,dt$ and the leader duty cycle. The
Monte Carlo estimator $\hat P(\mathrm{tube})$ is the fraction of
realizations satisfying all three terminal conditions and is therefore the
sample-average approximation of the chance constraint \eqref{eq:chance}.
Reporting the three component fractions separately makes it possible to
identify which part of the target tube is actually binding.

\subsection{Nominal-design comparison}
\label{subsec:num-nominal}

At the nominal design point
$(T,|\calL|,\tau_{\max})=(8.0\,\mathrm{s},\,8,\,0.25\,\mathrm{s})$,
results over $M=60$ realizations are reported in
\Cref{tab:numerics-feedback} and \Cref{fig:numerical_results}(a).

The sparse controller attains
$\hat P(\mathrm{tube})=0.950$ ($k=57/60$, 95\% Wilson CI $[0.86,\,0.98]$),
thereby meeting the prescribed reliability threshold of $0.95$. The
baseline controller achieves
$\hat P(\mathrm{tube})=0.917$ ($k=55/60$, 95\% Wilson CI $[0.82,\,0.96]$),
which is nominally below threshold. Because the confidence intervals overlap
substantially, we do not interpret the observed gap of $0.033$ as
statistically decisive evidence of superior reliability per se. At
this sample size, the reliability comparison is suggestive rather than
conclusive. The robust finding is instead that the sparse controller is the
one that reaches the design threshold, and that it does so with
substantially less actuation.

The component-wise fractions show that the formation requirement
$\calE_{\mathrm{form}}(T)\le \delta_f^{\,2}$ is the binding part of the tube
for both controllers, whereas the velocity and safety conditions are
satisfied in essentially every realization. In particular, the minimum
pairwise distance remains comfortably above $\rho=0.03$ in all runs of both
controllers. Thus, the finite-sample reliability difference at the nominal
point is driven by the formation criterion rather than by any incipient
collision risk. This interpretation is consistent with the sample means in
\Cref{tab:numerics-feedback}: the average terminal formation error
$5.8$--$6.0$ lies below $\delta_f^{\,2}=9.0$, but not so far below that
tail realizations cannot exceed the tolerance.

The actuation savings are quantitatively more decisive. Relative to the
baseline, the sparse controller reduces the mean $L^1$ control expenditure
from $41.99$ to $36.57$ (a reduction of $12.9\%$) and lowers the mean leader
duty cycle from $100\%$ to $71.3\%$. This is accompanied by a small
improvement in mean total cost, from $41.36$ to $40.34$. Unlike the
tube-membership probability, these are direct sample means of per-run
quantities and are therefore not subject to binomial proportion uncertainty.

These findings are consistent with the heuristic discussion in
\Cref{rem:heuristic} and with the sparse-stabilization literature
\cite{caponigro2013sparse,caponigro2015sparse}:
persistent leader intervention is not uniformly beneficial in a delayed
stochastic swarm. Once the passive alignment and interaction dynamics begin
to drive the population toward the target regime, thresholded disengagement
can reduce unnecessary corrective effort without degrading terminal
performance.

\begin{table}[ht]
\centering
\caption{Nominal-design comparison of the baseline and sparse feedback
leader controllers at
$(T,|\calL|,\tau_{\max})=(8.0\,\mathrm{s},\,8,\,0.25\,\mathrm{s})$
over $M=60$ independent realizations. Tolerances are given in
\eqref{eq:tolerances-num}. Intervals are 95\% Wilson score confidence
intervals. Bold marks the more favorable entry where the comparison is
clearly favorable.}
\label{tab:numerics-feedback}
\small
\begin{tabular}{@{}lcc@{}}
\toprule
Metric & Baseline & Sparse \\
\midrule
\multicolumn{3}{@{}l}{\emph{Tube-membership fractions (Monte Carlo)}}\\
$\hat P(\calE_{\mathrm{vel}}(T)\le\varepsilon_v^{\,2})$
  & 0.967 & 0.983 \\
$\hat P(\calE_{\mathrm{form}}(T)\le\delta_f^{\,2})$
  & 0.950 & 0.967 \\
$\hat P(\min_{i\neq j}|x_i(T)-x_j(T)|\ge\rho)$
  & 1.000 & 1.000 \\
$\hat P(\mathrm{tube})$
  & 0.917 & \textbf{0.950} \\
\quad $k/M$
  & 55/60 & 57/60 \\
\quad 95\% Wilson CI
  & $[0.82,\,0.96]$ & $[0.86,\,0.98]$ \\
\midrule
\multicolumn{3}{@{}l}{\emph{Terminal statistics (sample means)}}\\
$\calE_{\mathrm{vel}}(T)$
  & 0.2388 & 0.2365 \\
$\calE_{\mathrm{form}}(T)$
  & 5.9601 & \textbf{5.8012} \\
$\min_{i\neq j}|x_i(T)-x_j(T)|$
  & 0.1291 & 0.1185 \\
\midrule
\multicolumn{3}{@{}l}{\emph{Actuation}}\\
Mean $L^1$ control cost
  & 41.99 & \textbf{36.57} \\
Mean leader duty cycle
  & 100\% & \textbf{71.3\%} \\
\midrule
Mean total cost
  & 41.36 & \textbf{40.34} \\
\bottomrule
\end{tabular}
\end{table}

\begin{figure*}[t]
\centering
\includegraphics[width=0.95\linewidth]{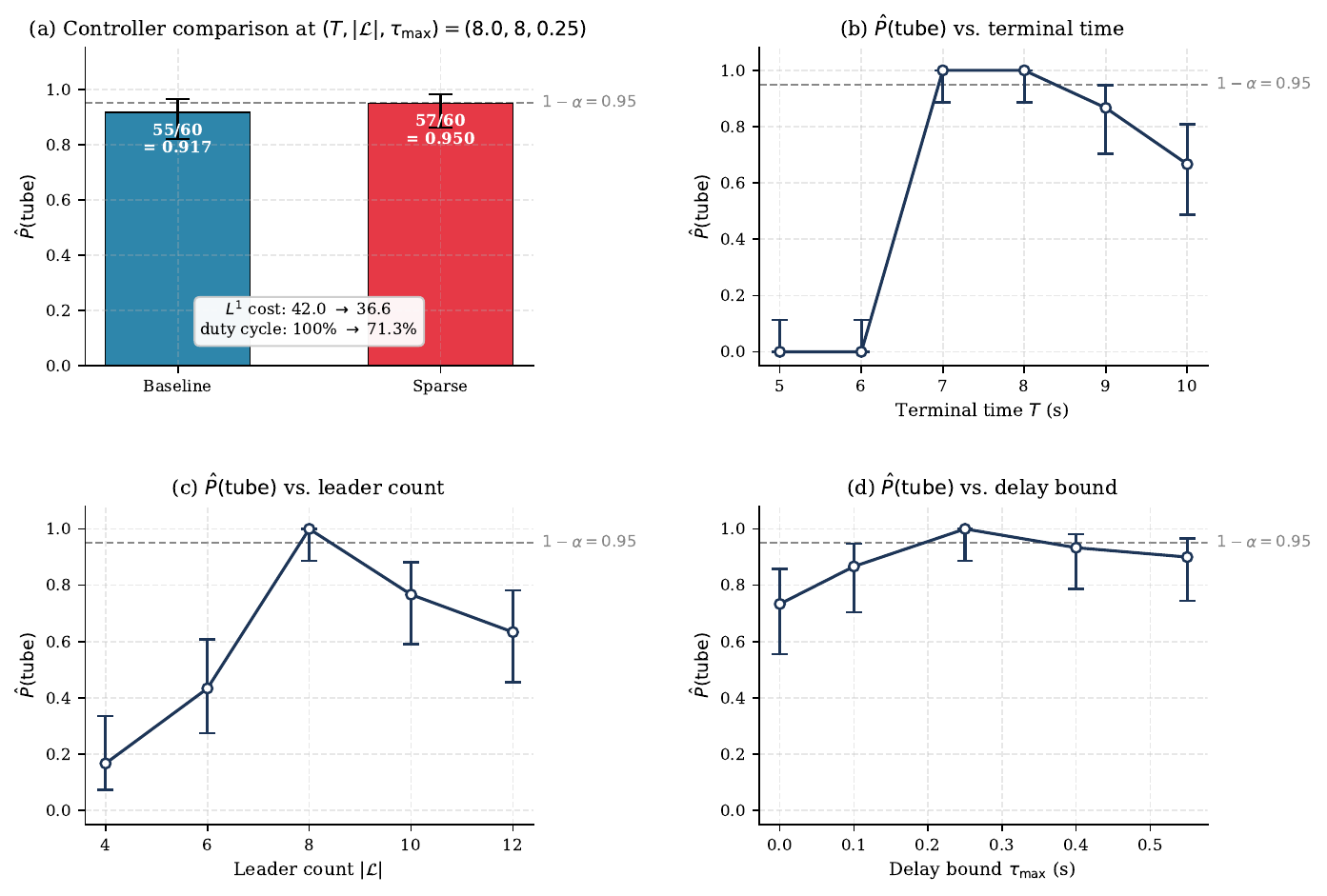}
\caption{Computational results for the finite-$N$ system.
(a) Controller comparison at the nominal design point
$(T,|\calL|,\tau_{\max})=(8.0\,\mathrm{s},\,8,\,0.25\,\mathrm{s})$,
$M=60$: bars show $\hat P(\mathrm{tube})$ with 95\% Wilson score confidence
intervals, and the inset reports the mean $L^1$ actuation cost and leader
duty cycle.
(b) Sparse-controller success probability across the terminal-time sweep
$T\in\{5,6,7,8,9,10\}\,\mathrm{s}$, $M=30$, showing the feasible window at
$T\in[7,8]\,\mathrm{s}$.
(c) Sensitivity to leader count
$|\calL|\in\{4,6,8,10,12\}$, $M=30$, showing a non-monotone dependence with
best tested performance at $|\calL|=8$.
(d) Sensitivity to the delay bound
$\tau_{\max}\in\{0.00,0.10,0.25,0.40,0.55\}\,\mathrm{s}$, $M=30$.
All error bars are 95\% Wilson score intervals, and the dashed line marks
the required reliability level $1-\alpha=0.95$.}
\label{fig:numerical_results}
\end{figure*}

A central feature of \Cref{prob:main} is that the terminal time $T$ is
itself a design variable. To illustrate why this matters, we sweep
$T\in\{5,6,7,8,9,10\}\,\mathrm{s}$ using the sparse controller, while
holding all other parameters at their nominal values and using $M=30$
realizations per point. The results are shown in
\Cref{fig:numerical_results}(b).

For short horizons $T\le 6.0\,\mathrm{s}$, the sparse controller fails
catastrophically:
$\hat P(\mathrm{tube})=0.000$ with 95\% Wilson CI $[0.00,\,0.12]$.
Component-wise inspection shows that both the velocity and formation
requirements fail in this regime: the swarm does not have enough time, under
pair-dependent delayed alignment, to contract velocity dispersion and
tighten the formation before the terminal time.

For intermediate horizons $T\in[7.0,8.0]\,\mathrm{s}$, every realization
satisfies the tube:
$\hat P(\mathrm{tube})=1.000$ with 95\% Wilson CI $[0.89,\,1.00]$.
Within this feasible window, the lowest mean total cost occurs at
$T=7.0\,\mathrm{s}$.

For longer horizons $T\ge 9.0\,\mathrm{s}$, the success probability declines
again; at $T=10.0\,\mathrm{s}$ we obtain
$\hat P(\mathrm{tube})=0.667$ with 95\% Wilson CI $[0.49,\,0.81]$.
Component-wise diagnosis indicates that the deterioration is primarily due
to the safety requirement $\min_{i\neq j}|x_i-x_j|\ge\rho$: over longer
horizons, nondegenerate stochastic forcing has enough time to erode
pairwise separations that were achieved earlier.

This sweep gives concrete support to the free-terminal-time formulation:
neither very short nor arbitrarily long horizons are optimal. The
intermediate window near $T=7$--$8\,\mathrm{s}$ is therefore a genuine
design outcome rather than a modeling convenience.

\subsection{Sensitivity to leader topology and communication delay}
\label{subsec:num-sensitivity}

We next assess the robustness of the sparse controller with respect to the
two structural parameters $|\calL|$ and $\tau_{\max}$. All sweeps use
$M=30$ realizations and keep the remaining parameters at their nominal
values.

\paragraph{Leader-density sensitivity.}
Varying the leader count over $|\calL|\in\{4,6,8,10,12\}$
(\Cref{fig:numerical_results}(c)) reveals a clearly non-monotone dependence
of $\hat P(\mathrm{tube})$ on the number of leaders. In the under-actuated
regime $|\calL|\le 6$, the control authority is insufficient to steer the
swarm into the target tube with high probability, and
$\hat P(\mathrm{tube})<0.5$. At the nominal choice $|\calL|=8$, the sparse
controller achieves $\hat P(\mathrm{tube})=1.000$ with 95\% Wilson interval
$[0.89,\,1.00]$. Increasing the number of leaders beyond this point does
not produce monotone improvement: for $|\calL|\in\{10,12\}$, the success
probability decreases again and the mean total cost rises. This is
consistent with the sparse-actuation philosophy of
\cite{caponigro2013sparse,caponigro2015sparse}: more direct actuation is not
automatically better, and leader density is itself a meaningful design
variable.

\paragraph{Delay sensitivity.}
Varying the delay bound over
$\tau_{\max}\in\{0.00,0.10,0.25,0.40,0.55\}\,\mathrm{s}$
(\Cref{fig:numerical_results}(d)), the empirical success probability is
lowest at the smallest tested delays and is at or above the reliability
threshold for $\tau_{\max}\ge 0.25\,\mathrm{s}$. We do not interpret
this pattern as evidence that larger delays are intrinsically beneficial.
Rather, it reflects the interaction of thresholded feedback, pair-dependent
delayed alignment, and common noise at the tuned operating point. The
robust conclusion is more modest: once tuned at the nominal design point,
the sparse controller remains computationally feasible over a nontrivial
range of delay bounds.

\subsection{Limitations and summaries of the numerical study}
\label{subsec:num-limits}

Several caveats are essential for a balanced interpretation of these
results.

First, all success probabilities are Monte Carlo estimates based on
$M\in\{30,60\}$ realizations. The Wilson intervals quantify the associated
sampling uncertainty, but they do not remove it. In particular, the
difference between the sparse and baseline tube probabilities at the
nominal point should be read as suggestive rather than definitive.

Second, the two controllers compared here share a fixed raw feedback
architecture and differ only in whether the signal is applied continuously
or thresholded. The sparse benchmark is not claimed to solve
\Cref{prob:main} optimally. Comparison with a genuine maximum-principle or
dynamic-programming solution is left for future work.

Third, the sensitivity studies are one-dimensional excursions from the
nominal operating point. Joint sweeps over $(|\calL|,\tau_{\max},T)$ could
reveal interactions that are not visible in the marginal curves shown here.

Fourth, the sample-average treatment of the chance constraint
\cite{nemirovski2007convex} is transparent and consistent,
but it is not sharp in a finite-sample sense. More refined scenario-based
or statistical-learning-style guarantees would require either additional
sampling effort or stronger structural assumptions.

The numerical study supports four main conclusions.
\begin{enumerate}[label=(\roman*),nosep]
\item At the nominal design point, the sparse controller reaches the
prescribed terminal chance level of $0.95$ and reduces mean actuation and
leader duty cycle relative to the continuously active baseline, although the
reliability gap itself is not statistically decisive at the present sample
sizes.
\item The terminal-time sweep reveals a genuine intermediate feasible window
near $T=7$--$8\,\mathrm{s}$, confirming the value of a free-terminal-time
formulation.
\item Leader density is a substantive design variable: too few leaders are
under-actuating, but additional leaders do not monotonically improve either
reliability or cost.
\item Delay sensitivity is nontrivial, yet the tuned sparse controller
remains feasible over a nontrivial tested range of pair-dependent delay
bounds.
\end{enumerate}

Taken together, these experiments are consistent with the broader structure
of the paper. The regularized delayed stochastic particle model is well
posed on the simulation horizon, the chance-constrained formulation captures
the relevant terminal design objective, and the sparse-control viewpoint is
supported by measurable reductions in actuation cost and duty cycle, with
reliability gains that are promising but not yet statistically decisive at
the sample sizes used here.
\section{Discussion and Conclusion}
\label{sec:discussion}

We have developed a finite-$N$ framework for delayed stochastic
leader--follower dynamics with pair-dependent communication delays,
singular repulsive interactions, and sparse bounded leader actuation.
The main analytical contribution is the explicit separation between the
exact discontinuous communication law and a regularized surrogate that is
compatible with a strong It\^o formulation. In this regularized setting,
the solution theory has a natural two-step structure: local strong
well-posedness on the collision-free region, followed by global extension
once collisions are excluded by a barrier argument. The key technical
device is an augmented Lyapunov functional that combines the
spatial moment, kinetic energy, and singular repulsive energy, thereby
providing both collision control and a Gr\"onwall estimate uniform in the
localization radius.

On the control side, we formulated a free-terminal-time,
chance-constrained sparse-leader problem that balances convergence speed,
actuation sparsity, and probabilistic safety. The numerical study supports
this formulation in a concrete way. At the nominal design point, the
sparse feedback benchmark satisfies the prescribed terminal reliability
threshold while using less total actuation and a substantially smaller
leader duty cycle than the continuously active baseline. The numerical
results also reveal a nontrivial feasible window in terminal time and a
non-monotone dependence on both leader density and delay. At the same
time, the Monte Carlo confidence intervals indicate that the reliability
advantage at the nominal point should be interpreted as suggestive rather
than statistically decisive at the sample sizes used here; the most robust
numerical conclusion is the reduction in actuation cost achieved by sparse
feedback.

Several directions remain open. A first is to derive explicit
model-dependent no-collision criteria for concrete repulsive potentials.
A second is to study the delayed chance-constrained control problem by
stochastic maximum-principle or dynamic-programming methods, beyond the
benchmark sparse feedback laws used here. These questions would sharpen both the
analytical and computational aspects of sparse control in delayed
stochastic swarms.
\appendix
\section{Appendix A: Local Well-Posedness Proof}
\label{app:local-proof}
This appendix provides a proof of \Cref{thm:local}. We work on the fixed
time interval $[-\tau_{\max},T]$, where $T>0$ is arbitrary, and for a fixed
admissible control $u\in\calU_T$.

The argument follows the standard localization strategy for stochastic
functional differential equations with bounded memory: one first verifies
that, on bounded collision-free sets, all coefficients are locally Lipschitz
in the current state and the delayed segment; one then applies a local
existence--uniqueness and continuation theorem for stochastic functional
differential equations, and finally identifies the maximal interval of
existence with the collision-free interval. Such local
existence--uniqueness and continuation results for SFDEs with bounded
delay are standard; see \cite{mohammed1984stochastic,mao2007stochastic,von2010existence}.
\subsection*{A.1. Reformulation as a stochastic functional differential equation}
Let
\[
  Y(t):=
  \bigl(x_1(t),\dots,x_N(t),v_1(t),\dots,v_N(t)\bigr)
  \in\R^{2dN},
\]
and denote by
\[
  Y_t(\theta):=Y(t+\theta),
  \qquad \theta\in[-\tau_{\max},0],
\]
the associated history segment. We write the history space as
\[
  \mathcal C:=C([-\tau_{\max},0];\R^{2dN}),
\]
equipped with the supremum norm
\[
  \|\varphi\|_\infty
  :=\sup_{\theta\in[-\tau_{\max},0]} |\varphi(\theta)|.
\]
For $\varphi=(\varphi^x,\varphi^v)\in\mathcal C$, with
\begin{align*}
  \varphi^x(\theta)&=\bigl(\varphi^x_1(\theta),\dots,\varphi^x_N(\theta)\bigr), \\
  \varphi^v(\theta)&=\bigl(\varphi^v_1(\theta),\dots,\varphi^v_N(\theta)\bigr),
\end{align*}
define the drift $F:[0,T]\times\mathcal C\to\R^{2dN}$ componentwise by
\begin{align*}
  F(t,\varphi)
  := \bigl(&F_1^x(t,\varphi),\dots,F_N^x(t,\varphi), \\
           &F_1^v(t,\varphi),\dots,F_N^v(t,\varphi)\bigr),
\end{align*}
where
\[
  F_i^x(t,\varphi):=\varphi_i^v(0),
\]
and
\begin{equation}\label{eq:appendix-drift}
\begin{aligned}
  F_i^v(t,\varphi)
  :=\;&
  \mathcal A_i^{\varepsilon,\delta}\bigl(t,\varphi^x(0),\varphi^v\bigr) \\
  &- \frac1N \sum_{j\neq i} \nabla U\bigl(\varphi_i^x(0)-\varphi_j^x(0)\bigr) \\
  &- \nabla V_{\mathrm{obs}}\bigl(\varphi_i^x(0)\bigr)
  - b_i\bigl(\varphi_i^v(0)-\vstar\bigr) \\
  &+ B_i u_i(t).
\end{aligned}
\end{equation}
Likewise, define the diffusion map
$G:[0,T]\times\mathcal C\to\R^{2dN\times(Nr+r_0)}$ by
\[
  G(t,\varphi)
  =
  \begin{pmatrix}
    0 \\
    \Sigma(t,\varphi)
  \end{pmatrix},
\]
where, letting $\mu_\varphi := \mu(\varphi^x(0),\varphi^v(0))$, the $i$-th
velocity block of $\Sigma$ is given by
\begin{equation*}
\begin{aligned}
  \Sigma_i(t,\varphi)
  := \Bigl(& 0,\dots,0,\, \sigma_i\bigl(\varphi_i^x(0),\varphi_i^v(0),\mu_\varphi\bigr), \\
  & 0,\dots,0,\, \sigma_i^0\bigl(\varphi_i^x(0),\varphi_i^v(0),\mu_\varphi\bigr) \Bigr),
\end{aligned}
\end{equation*}
and
\[
  \mu(\varphi^x(0),\varphi^v(0))
  :=
  \frac1N\sum_{k=1}^N \delta_{(\varphi_k^x(0),\varphi_k^v(0))}.
\]
Then \eqref{eq:pos}--\eqref{eq:vel} can be written compactly as the
stochastic functional differential equation
\begin{equation}\label{eq:appendix-sfde}
  dY(t)=F(t,Y_t)\,dt+G(t,Y_t)\,d\mathcal W(t),
\end{equation}
where $\mathcal W:=(W_1,\dots,W_N,W^0)$ is an $\R^{Nr+r_0}$-valued
Brownian motion.
\subsection*{A.2. Local Lipschitz continuity and local growth bounds on bounded collision-free sets}
Fix $R>0$ and $\kappa>0$, and define the bounded collision-free set
\begin{equation*}
\begin{aligned}
  \mathcal C_{R,\kappa}
  := \Bigl\{ \varphi\in\mathcal C:\,
    \|\varphi\|_\infty\le R,\; 
    |\varphi_i^x(0)-\varphi_j^x(0)|\ge \kappa
    \text{ for all }i\neq j \Bigr\}.
\end{aligned}
\end{equation*}
We claim that $F(t,\cdot)$ and $G(t,\cdot)$ are uniformly Lipschitz on
$\mathcal C_{R,\kappa}$, uniformly in $t\in[0,T]$. Note that, although
$\mathcal C_{R,\kappa}$ is bounded but not compact in $\mathcal C$, the
relevant coefficients depend on $\varphi$ only through finitely many
finite-dimensional evaluations (the present state and delayed velocity
samples), so the local-Lipschitz hypotheses in Assumptions~\ref{ass:net}, \ref{ass:align}, \ref{ass:pot}, and~\ref{ass:noise} apply.
\paragraph{Position component.}
Since $F_i^x(t,\varphi)=\varphi_i^v(0)$, we immediately have
\[
  |F_i^x(t,\varphi)-F_i^x(t,\psi)|
  \le \|\varphi-\psi\|_\infty.
\]
\paragraph{Alignment component.}
By \Cref{ass:align}, for each $i$ the map
$(t,\mathbf{x},\mathbf{v}_t)\mapsto \mathcal A_i^{\varepsilon,\delta}(t,\mathbf{x},\mathbf{v}_t)$
is locally Lipschitz on compact collision-free $\mathbf{x}$ and bounded $\mathbf{v}_t$.
Since $\varphi, \psi \in \mathcal C_{R,\kappa}$, their evaluations map into
compact collision-free sets in the state space and bounded sets in the
history space. Therefore, there exists $L_{A}(R,\kappa,T)>0$ such that
\begin{equation*}
\begin{aligned}
  &\bigl| \mathcal A_i^{\varepsilon,\delta}(t,\varphi^x(0),\varphi^v)
    - \mathcal A_i^{\varepsilon,\delta}(t,\psi^x(0),\psi^v) \bigr| \\
  &\quad \le L_A(R,\kappa,T)\,\|\varphi-\psi\|_\infty.
\end{aligned}
\end{equation*}
\paragraph{Potential terms.}
On the set
\begin{equation*}
  \Bigl\{ \mathbf{x}\in(\R^d)^N:\, |\mathbf{x}|\le R,\;
    \min_{i\neq j}|x_i-x_j|\ge\kappa \Bigr\},
\end{equation*}
the map
$\mathbf{x}\mapsto \frac1N\sum_{j\neq i}\nabla U(x_i-x_j)$
is Lipschitz. Indeed, by \Cref{ass:pot}, $\nabla U_{\mathrm{form}}$ is
locally Lipschitz on $\R^d$, while $\nabla U_{\mathrm{rep}}$ is $C^1$ on
$\R^d\setminus\{0\}$; hence $\nabla U$ is Lipschitz on every compact subset
of $\R^d\setminus\{0\}$. Since pairwise distances are bounded below by
$\kappa$, the arguments $x_i-x_j$ remain in such a compact set. Therefore,
for some $L_U(R,\kappa)>0$,
\begin{equation*}
\begin{aligned}
  &\Bigl| \frac1N\sum_{j\neq i}\nabla U(\varphi_i^x(0)-\varphi_j^x(0)) \\
  &\quad - \frac1N\sum_{j\neq i}\nabla U(\psi_i^x(0)-\psi_j^x(0)) \Bigr|
  \le L_U(R,\kappa)\,\|\varphi-\psi\|_\infty.
\end{aligned}
\end{equation*}
Similarly, since $\nabla V_{\mathrm{obs}}$ is locally Lipschitz,
\begin{equation*}
\begin{aligned}
  |\nabla V_{\mathrm{obs}}(\varphi_i^x(0))
   -\nabla V_{\mathrm{obs}}(\psi_i^x(0))|
  \le L_V(R)\,\|\varphi-\psi\|_\infty.
\end{aligned}
\end{equation*}
\paragraph{Pinning and control terms.}
The pinning term satisfies
\begin{equation*}
  \bigl| b_i(\varphi_i^v(0)-\vstar)-b_i(\psi_i^v(0)-\vstar) \bigr|
  \le |b_i|\,\|\varphi-\psi\|_\infty.
\end{equation*}
The control term $B_i u_i(t)$ is independent of the state and therefore
does not affect Lipschitz continuity.

Combining the previous bounds yields a constant $L_F(R,\kappa,T)>0$ such
that
$|F(t,\varphi)-F(t,\psi)| \le L_F(R,\kappa,T)\|\varphi-\psi\|_\infty$
for all $t\in[0,T]$ and $\varphi,\psi\in\mathcal C_{R,\kappa}$.
\paragraph{Diffusion term.}
For the empirical measure, the coupling induced by the identity map gives
\begin{equation*}
\begin{aligned}
  &W_2\bigl(\mu(\varphi^x(0),\varphi^v(0)),
           \mu(\psi^x(0),\psi^v(0))\bigr)^2 \\
  &\quad \le \frac1N \sum_{k=1}^N \Bigl( |\varphi_k^x(0)-\psi_k^x(0)|^2
    + |\varphi_k^v(0)-\psi_k^v(0)|^2 \Bigr) \\
  &\quad = \frac{1}{N}|\varphi(0)-\psi(0)|^2
  \le \|\varphi-\psi\|_\infty^2.
\end{aligned}
\end{equation*}
Hence, by \Cref{ass:noise}, there exists $L_G(R)>0$ such that
\[
  \|G(t,\varphi)-G(t,\psi)\|
  \le L_G(R)\,\|\varphi-\psi\|_\infty,
\]
for all $t\in[0,T]$ and $\varphi,\psi\in\mathcal C_{R,\kappa}$.

We have therefore shown that on every bounded collision-free set
$\mathcal C_{R,\kappa}$, both $F$ and $G$ are uniformly Lipschitz in the
history variable.

On $\mathcal C_{R,\kappa}$, the above arguments also imply boundedness of
the drift and diffusion. Indeed, the regularized alignment operator is
bounded on bounded sets by \Cref{ass:align}; the potential terms are
bounded on collision-free compact sets by \Cref{ass:pot}; and the pinning
and control terms are bounded because $\varphi\in\mathcal C_{R,\kappa}$
and $u\in\calU_T$. For the diffusion, \Cref{ass:noise} gives local
boundedness through the local Lipschitz property plus linear growth.
Consequently, on every $\mathcal C_{R,\kappa}$ there exists
$C_{R,\kappa,T}>0$ such that
\[
  |F(t,\varphi)| + \|G(t,\varphi)\|
  \le C_{R,\kappa,T},
\]
for all $t\in[0,T]$ and $\varphi\in\mathcal C_{R,\kappa}$.
\subsection*{A.3. Localized equation and stopping times}
Let
\begin{equation*}
\begin{aligned}
  \tau_{R,\kappa}  := \inf\Bigl\{ t\ge 0:\, |Y(t)|\ge R \ \text{or} \
  \min_{i\neq j}|x_i(t)-x_j(t)|\le \kappa \Bigr\} \wedge T.
\end{aligned}
\end{equation*}
On the stochastic interval $[-\tau_{\max},\tau_{R,\kappa}]$, the segment
process $Y_t$ remains in $\mathcal C_{R,\kappa}$. By the local Lipschitz
and local boundedness established above, the coefficients of
\eqref{eq:appendix-sfde} satisfy the hypotheses of the standard local
existence--uniqueness theorem for SFDEs with bounded delay
\cite{mohammed1984stochastic,mao2007stochastic,von2010existence}. Therefore, for
each $R,\kappa$ there exists a unique strong solution up to
$\tau_{R,\kappa}$. Standard continuation then yields a unique maximal
strong solution up to the limit
\[
  \tau_\ast
  :=
  \sup_{R\in\mathbb N,\ \kappa\downarrow 0}\tau_{R,\kappa}.
\]
The continuation principle for SFDEs states that either $\tau_\ast=T$, or
the solution exits every bounded collision-free set as
$t\uparrow\tau_\ast$. This is the SFDE analogue of the usual continuation
theorem for SDEs
\cite{von2010existence,mao2007stochastic,mohammed1984stochastic}.
\subsection*{A.5. Identification of the maximal interval}
We now show that $\tau_\ast = T\wedge \tau_{\mathrm{coll}}$. Fix
$\kappa>0$ and let $\sigma_\kappa := \inf\{t:\min_{i\neq j}|x_i(t)-x_j(t)|\le\kappa\}$.
On $[0,\sigma_\kappa]$, the repulsive gradient $\nabla U_{\mathrm{rep}}$ is
bounded (since pairwise distances are bounded below by $\kappa$), and by Assumptions~\ref{ass:net}, \ref{ass:align}, \ref{ass:pot} and~\ref{ass:noise} the remaining drift and the
diffusion have at most linear growth in $(\mathbf{x},\mathbf{v})$ (the control is
bounded by assumption). Standard linear-growth theory for stochastic
functional differential equations \cite{mao2007stochastic}
then excludes explosion of the state norm $|Y(t)|$ on $[0,\sigma_\kappa]$.
Letting $\kappa\downarrow 0$, we obtain $\sigma_\kappa\uparrow\tau_{\mathrm{coll}}$
almost surely, and hence $\tau_\ast=T\wedge\tau_{\mathrm{coll}}$. This
proves that the system admits a pathwise unique strong solution on
$[-\tau_{\max},\,T\wedge\tau_{\mathrm{coll}})$.
Since $T>0$ was arbitrary, the proof of \Cref{thm:local} is complete.
\hfill$\square$

The proof above is deliberately organized around bounded collision-free
localization sets because this is the natural framework for the present
model: the regularized communication law and the delayed alignment term
are smooth on such sets, while the only genuine singularity in the drift
arises from the repulsive interaction at collisions. This is precisely why
the local theory is stopped at $\tau_{\mathrm{coll}}$, and why the global
theorem requires the separate barrier argument developed in
\Cref{sec:collision-analysis} and Appendix~\ref{app:collision-proof}.
\section{Appendix B: Collision-Avoidance Proof}
\label{app:collision-proof}
This appendix provides a rigorous collision-avoidance argument consistent
with the local theory of \Cref{thm:local}. The key point is that the
singular repulsive energy alone does not control the absolute spatial size
of the configuration. Accordingly, we work with an augmented Lyapunov
functional containing the second spatial moment, the kinetic energy, and
the singular repulsive energy. This allows us to derive a localized
Gr\"onwall estimate with a constant independent of the localization
radius. The strategy follows the deterministic Cucker--Smale
collision-avoidance literature
\cite{cucker2010avoiding,park2010cucker,carrillo2017sharp}, adapted to
the stochastic setting via Lyapunov methods in the style of
\cite{khasminskii2011stochastic,mao2007stochastic}.

Throughout, fix $T>0$ and an admissible control $u\in\calU_T$, and let
\[
  (\mathbf{x},\mathbf{v})=
  (x_1,\dots,x_N,v_1,\dots,v_N)
\]
denote the local strong solution of \eqref{eq:pos}--\eqref{eq:vel} on
$[-\tau_{\max},\,T\wedge\tau_{\mathrm{coll}})$.
\subsection*{B.1. Augmented Lyapunov functional}
For $t<\tau_{\mathrm{coll}}$, define
\begin{equation}\label{eq:appB-Htilde}
\begin{aligned}
  \widetilde{\mathscr H}(t)
  :=\;&
  \frac12\sum_{i=1}^N |x_i(t)|^2
  + \frac12\sum_{i=1}^N |v_i(t)|^2 \\
  &+ \frac1N\sum_{1\le i<j\le N}
    U_{\mathrm{rep}}\bigl(x_i(t)-x_j(t)\bigr).
\end{aligned}
\end{equation}
Assuming without loss of generality that $U_{\mathrm{rep}} \ge 0$ (which
can be arranged by adding a constant when bounded below), and since
$U_{\mathrm{rep}}(z)\to+\infty$ as $|z|\downarrow 0$, the functional
$\widetilde{\mathscr H}$ is coercive with respect to collisions: if
$\min_{i\neq j}|x_i(t)-x_j(t)|\downarrow 0$, then
\[
  \widetilde{\mathscr H}(t)\to+\infty.
\]
Moreover, $\widetilde{\mathscr H}$ controls both the spatial and kinetic
moments:
\begin{equation}\label{eq:appB-coercive}
  \frac12|\mathbf{x}(t)|^2+\frac12|\mathbf{v}(t)|^2
  \le \widetilde{\mathscr H}(t),
  \qquad t<\tau_{\mathrm{coll}}.
\end{equation}
Introduce the localization sequence
\begin{equation}\label{eq:appB-tauR}
\begin{aligned}
  \tau_R
  :=
  \inf\Bigl\{
    t\in[0,T\wedge\tau_{\mathrm{coll}}):
    \widetilde{\mathscr H}(t)\ge R
  \Bigr\},
  \qquad R\ge 1.
\end{aligned}
\end{equation}
By \eqref{eq:appB-coercive}, on $[0,\tau_R]$ both $|\mathbf{x}|$ and $|\mathbf{v}|$
are bounded by a constant depending only on $R$, and the configuration
remains collision-free.

We compute the evolution of $\widetilde{\mathscr H}$ on $[0,\tau_R]$.
Since $dx_i(t)=v_i(t)\,dt$, the spatial moment satisfies
\begin{equation}\label{eq:appB-position-moment}
  d\Bigl(\frac12|x_i(t)|^2\Bigr)
  =
  x_i(t)\cdot v_i(t)\,dt.
\end{equation}
Because the noise enters only through the velocity equation, the
repulsive potential is differentiated by the ordinary chain rule:
\begin{equation}\label{eq:appB-rep-chain}
\begin{aligned}
  &d\Bigl[
    \frac1N\sum_{1\le i<j\le N}
    U_{\mathrm{rep}}\bigl(x_i(t)-x_j(t)\bigr)
  \Bigr] \\
  &\quad =
  \frac1N\sum_{1\le i<j\le N}
  \nabla U_{\mathrm{rep}}\bigl(x_i(t)-x_j(t)\bigr)\cdot
  \bigl(v_i(t)-v_j(t)\bigr)\,dt.
\end{aligned}
\end{equation}
Applying It\^o's formula to $\frac12|v_i(t)|^2$ gives
\begin{equation}\label{eq:appB-kinetic}
\begin{aligned}
  &d\Bigl(\frac12|v_i(t)|^2\Bigr) \\
  &= v_i(t)\cdot dv_i(t)  + \frac12
  \operatorname{tr}\!\Bigl(
    \sigma_i\sigma_i^\top
    + \sigma_i^0(\sigma_i^0)^\top
  \Bigr)\bigl(x_i(t),v_i(t),\mu_t^N\bigr)\,dt \\
  &= v_i(t)\cdot
  \Biggl[
    \mathcal A_i^{\varepsilon,\delta}\bigl(t,\mathbf{x}(t),\mathbf{v}_t\bigr) 
    - \frac1N\sum_{j\neq i}\nabla U\bigl(x_i(t)-x_j(t)\bigr) \\
  &\qquad\qquad
    - \nabla V_{\mathrm{obs}}\bigl(x_i(t)\bigr)
    - b_i\bigl(v_i(t)-\vstar\bigr)
    + B_i u_i(t)
  \Biggr]dt \\
  &\quad
  + v_i(t)\cdot
    \sigma_i\bigl(x_i(t),v_i(t),\mu_t^N\bigr)\,dW_i(t) \\
  &\quad
  + v_i(t)\cdot
    \sigma_i^0\bigl(x_i(t),v_i(t),\mu_t^N\bigr)\,dW^0(t) \\
  &\quad
  + \frac12
  \operatorname{tr}\!\Bigl(
    \sigma_i\sigma_i^\top
    + \sigma_i^0(\sigma_i^0)^\top
  \Bigr)\bigl(x_i(t),v_i(t),\mu_t^N\bigr)\,dt.
\end{aligned}
\end{equation}
Summing \eqref{eq:appB-position-moment}, \eqref{eq:appB-rep-chain}, and
\eqref{eq:appB-kinetic} over $i$, the contribution of the singular
repulsive force cancels exactly:
\begin{equation}\label{eq:appB-cancel}
\begin{aligned}
  &-\sum_{i=1}^N
    v_i(t)\cdot \frac1N\sum_{j\neq i}
    \nabla U_{\mathrm{rep}}\bigl(x_i(t)-x_j(t)\bigr) \\
  &\quad +
  \frac1N\sum_{1\le i<j\le N}
  \nabla U_{\mathrm{rep}}\bigl(x_i(t)-x_j(t)\bigr)\cdot
  \bigl(v_i(t)-v_j(t)\bigr) \\
  &\quad = 0,
\end{aligned}
\end{equation}
where we use that $U_{\mathrm{rep}}$ is radial, hence even, and therefore
$\nabla U_{\mathrm{rep}}(-z)=-\nabla U_{\mathrm{rep}}(z)$. (The smooth
formation potential $U_{\mathrm{form}}$ does not cancel and remains
in the drift $\widetilde\Gamma$ below.)
Consequently,
\begin{equation}\label{eq:appB-Htilde-identity}
\begin{aligned}
  \widetilde{\mathscr H}(t\wedge\tau_R)
  =\;
  \widetilde{\mathscr H}(0)
  + \int_0^{t\wedge\tau_R}\widetilde{\Gamma}(s)\,ds
  + \widetilde{M}_R(t),
\end{aligned}
\end{equation}
where $\widetilde{M}_R$ is the local martingale
\begin{equation}\label{eq:appB-mart}
\begin{aligned}
  \widetilde{M}_R(t)
  :=\;&
  \sum_{i=1}^N\int_0^{t\wedge\tau_R}
  v_i(s)\cdot
  \sigma_i\bigl(x_i(s),v_i(s),\mu_s^N\bigr)\,dW_i(s) \\
  &+ \sum_{i=1}^N\int_0^{t\wedge\tau_R}
  v_i(s)\cdot
  \sigma_i^0\bigl(x_i(s),v_i(s),\mu_s^N\bigr)\,dW^0(s),
\end{aligned}
\end{equation}
and the drift term is
\begin{equation}\label{eq:appB-drift}
\begin{aligned}
  \widetilde{\Gamma}(t)
  :=\;&
  \sum_{i=1}^N x_i(t)\cdot v_i(t) + \sum_{i=1}^N
    v_i(t)\cdot\mathcal A_i^{\varepsilon,\delta}\bigl(t,\mathbf{x}(t),\mathbf{v}_t\bigr) \\
  &- \frac1N\sum_{i=1}^N\sum_{j\neq i}
    v_i(t)\cdot\nabla U_{\mathrm{form}}\bigl(x_i(t)-x_j(t)\bigr) \\
  &- \sum_{i=1}^N
    v_i(t)\cdot\nabla V_{\mathrm{obs}}\bigl(x_i(t)\bigr) - \sum_{i=1}^N
    b_i\,v_i(t)\cdot\bigl(v_i(t)-\vstar\bigr) \\
  &+ \sum_{i=1}^N
    v_i(t)\cdot B_i u_i(t) \\
  &+ \frac12\sum_{i=1}^N
  \operatorname{tr}\!\Bigl(
    \sigma_i\sigma_i^\top+\sigma_i^0(\sigma_i^0)^\top
  \Bigr)\bigl(x_i(t),v_i(t),\mu_t^N\bigr).
\end{aligned}
\end{equation}

The global result is conditional on the following estimate.
\begin{assumption}[Drift control for the augmented Lyapunov functional]
\label{ass:appB-drift}
For every $T>0$, there exists a deterministic constant $C_T>0$,
independent of $R$, such that
\begin{equation}\label{eq:appB-drift-bound}
  \widetilde{\Gamma}(t)
  \le C_T\bigl(1+\widetilde{\mathscr H}(t)\bigr)
\end{equation}
for all $t\le T\wedge\tau_R$ almost surely.
\end{assumption}
\begin{remark}
This is the precise model-dependent input needed for collision avoidance.
It encodes, in a single inequality, the required interaction between the
delayed alignment term, the smooth formation and obstacle forces, the
pinning and control terms, and the diffusion growth; compare
\cite{khasminskii2011stochastic,mao2007stochastic}.
\end{remark}
\subsection*{B.2. Uniform expectation bound}
Under \Cref{ass:appB-drift}, the stopped local martingale
$\widetilde{M}_R$ is a true martingale on $[0,T]$, because all
coefficients are bounded on $[0,\tau_R]$. Taking expectations in
\eqref{eq:appB-Htilde-identity} yields
\begin{equation*}
\begin{aligned}
  &\E[\widetilde{\mathscr H}(t\wedge\tau_R)] \\
  &\quad \le
  \widetilde{\mathscr H}(0)
  + C_T\int_0^t
  \bigl(1+\E[\widetilde{\mathscr H}(s\wedge\tau_R)]\bigr)\,ds,
\end{aligned}
\end{equation*}
for $0\le t\le T$. By Gr\"onwall's lemma,
\begin{equation}\label{eq:appB-uniform}
  \sup_{R\ge 1}\sup_{0\le t\le T}
  \E[\widetilde{\mathscr H}(t\wedge\tau_R)]
  <\infty.
\end{equation}
Now let $R\to\infty$. As argued in \Cref{subsec:augmented-no-collision},
$\tau_R\uparrow T\wedge\tau_{\mathrm{coll}}$ almost surely, and since
$\widetilde{\mathscr H}\ge 0$, Fatou's lemma gives
\begin{equation}\label{eq:appB-fatou}
\begin{aligned}
  \E\Bigl[\liminf_{R\to\infty}\widetilde{\mathscr H}(t\wedge\tau_R)\Bigr]
  &\le
  \liminf_{R\to\infty}
  \E[\widetilde{\mathscr H}(t\wedge\tau_R)]
  <\infty,
\end{aligned}
\end{equation}
for $0\le t\le T$.

Assume, by contradiction, that $\Prob(\tau_{\mathrm{coll}}\le T)>0$. Then
on the event $\{\tau_{\mathrm{coll}}\le T\}$, there exist indices
$i\neq j$ such that
\[
  |x_i(t)-x_j(t)|\downarrow 0
  \qquad\text{as }t\uparrow\tau_{\mathrm{coll}}.
\]
Since $U_{\mathrm{rep}}(z)\to+\infty$ as $|z|\downarrow 0$, it follows
from \eqref{eq:appB-Htilde} that
$\widetilde{\mathscr H}(t)\to+\infty$ as
$t\uparrow\tau_{\mathrm{coll}}$ on that event. Consequently,
$\widetilde{\mathscr H}(t\wedge\tau_R)\to+\infty$ as $R\to\infty$ on a set
of strictly positive probability, contradicting \eqref{eq:appB-fatou}.
Hence $\Prob(\tau_{\mathrm{coll}}\le T)=0$. Since $T>0$ was arbitrary,
\begin{equation}\label{eq:appB-nocollision}
  \Prob(\tau_{\mathrm{coll}}=+\infty)=1.
\end{equation}
This proves the no-collision property required in \Cref{thm:global}.
\hfill$\square$
\subsection*{B.3. Verification mechanism for \Cref{ass:appB-drift}}
We finally indicate how \eqref{eq:appB-drift-bound} is checked in
practice. On $[0,\tau_R]$, we have $\widetilde{\mathscr H}(t)\le R$, so by
\eqref{eq:appB-coercive},
\[
  |\mathbf{x}(t)|^2+|\mathbf{v}(t)|^2
  \le 2\,\widetilde{\mathscr H}(t).
\]
Hence every term in \eqref{eq:appB-drift} can be controlled by
$1+\widetilde{\mathscr H}(t)$, provided the smooth forces and diffusion
coefficients satisfy global linear-growth bounds \cite{mao2007stochastic,khasminskii2011stochastic}.
For example,
\[
  \sum_{i=1}^N x_i\cdot v_i
  \le \frac12|\mathbf{x}|^2+\frac12|\mathbf{v}|^2
  \le \widetilde{\mathscr H}(t),
\]
and similarly
\begin{equation*}
\begin{aligned}
  \Bigl|
    \sum_{i=1}^N v_i\cdot B_i u_i
  \Bigr|
  &\le C\bigl(1+|\mathbf{v}|^2\bigr) \\
  &\le C\bigl(1+\widetilde{\mathscr H}(t)\bigr),
\end{aligned}
\end{equation*}
because $u\in\calU_T$ is bounded. The alignment term is treated
analogously, using the linear-growth bound on
$\mathcal A_i^{\varepsilon,\delta}$ recorded in \Cref{ass:align}, and the
diffusion trace term is controlled by the linear-growth assumption on
$\sigma_i,\sigma_i^0$ in \Cref{ass:noise}. Therefore, once the model is
specified so that these coefficients have at most linear growth in
$(\mathbf{x},\mathbf{v})$, the drift estimate \eqref{eq:appB-drift-bound} follows.
\begin{remark}
The analytical core is the inclusion of the spatial moment
$\frac12\sum_i |x_i|^2$ in the Lyapunov functional. Without that term,
the energy does not control absolute spatial motion and the Gr\"onwall
estimate does not close; see the discussion in
\Cref{rem:why-augmented}.
\end{remark}

\section{Appendix C: Notation Table}
\label{app:notation}
\begin{table}[htbp]
\centering
\caption{Principal notation used throughout the paper, with the
regularization parameters and the tube tolerances now typographically
distinct.}
\label{tab:notation}
\small
\begin{tabular}{@{}l p{0.65\columnwidth}@{}}
\toprule
Symbol & Meaning \\
\midrule
$N$ & Number of agents \\
$d$ & Spatial dimension \\
$\mathbf{x}=(x_1,\dots,x_N)\in(\R^d)^N$ & Joint position configuration \\
$\mathbf{v}=(v_1,\dots,v_N)\in(\R^d)^N$ & Joint velocity configuration \\
$(x_i,v_i)$ & Position--velocity state of agent $i$ \\
$\mathbf{v}_t$ & Velocity history segment on $[-\tau_{\max},0]$ \\
$\calL$ & Leader set \\
$K$ & Nominal topological interaction number \\
$\tau_{ij}(t)$ & Pair-dependent communication delay from $j$ to $i$ \\
$\tau_{\max}$ & Uniform upper bound on all delays \\
$\chi_{ij}^{\mathrm{ex}}$ & Exact communication availability indicator \\
$r_{ij}^{\mathrm{ex}}$ & Exact topological rank of $j$ relative to $i$ \\
\midrule
\multicolumn{2}{@{}l}{\emph{Regularization parameters (dynamics)}}\\
$\varepsilon$ & Rank/availability regularization parameter in the
communication law \eqref{eq:weight} \\
$\delta$ & Denominator regularization in the alignment operator
\eqref{eq:alignop} \\
$\chi_{ij}^{\varepsilon}$ & Regularized availability surrogate \\
$r_{ij}^{\varepsilon}$ & Regularized topological-rank surrogate \\
$\phi$ & Bounded communication profile \\
$a_{ij}^{\varepsilon}$ & Regularized communication weight \\
$\eta_i^\varepsilon$ & Total regularized communication weight of agent $i$ \\
$\mathcal A_i^{\varepsilon,\delta}$ & Regularized normalized delayed
alignment operator \\
\midrule
\multicolumn{2}{@{}l}{\emph{Interaction, noise, well-posedness}}\\
$U_{\mathrm{rep}}$ & Singular repulsive interaction potential \\
$U_{\mathrm{form}}$ & Smooth formation interaction potential \\
$U=U_{\mathrm{rep}}+U_{\mathrm{form}}$ & Total pairwise interaction potential \\
$V_{\mathrm{obs}}$ & Obstacle/environmental potential \\
$\vstar$ & Target migration/reference velocity \\
$b_i$ & Velocity-pinning coefficient \\
$B_i$ & Actuation matrix for agent $i$ \\
$u_i$ & Control input of leader $i$ \\
$M_i$ & Actuator bound for leader $i$ \\
$\calU_T$ & Admissible control class on $[0,T]$ \\
$W_i$ & Idiosyncratic Brownian motion of agent $i$ \\
$W^0$ & Common Brownian motion \\
$\sigma_i$ & Idiosyncratic diffusion coefficient \\
$\sigma_i^0$ & Common-noise diffusion coefficient \\
$\mu_t^N$ & Empirical measure of the $N$-agent system at time $t$ \\
$\tau_{\mathrm{coll}}$ & First collision time \\
$\mathsf{Conf}_N(\R^d)$ & Collision-free configuration space \\
$\widetilde{\mathscr H}(t)$ & Augmented Lyapunov functional
\eqref{eq:Htilde-body} \\
\midrule
\multicolumn{2}{@{}l}{\emph{Control-problem tolerances (terminal design)}}\\
$\varepsilon_v$ & Terminal velocity tolerance in the target tube
\eqref{eq:target} \\
$\delta_f$ & Terminal formation tolerance in the target tube \\
$\rho$ & Minimum separation threshold in the target tube \\
$\calT_{\varepsilon_v,\delta_f,\rho}$ & Terminal target tube \\
$\alpha$ & Risk level in the terminal chance constraint \\
\midrule
\multicolumn{2}{@{}l}{\emph{Performance metrics and objectives}}\\
$\calE_{\mathrm{vel}}$ & Velocity-tracking error functional \\
$\calE_{\mathrm{form}}$ & Formation-error functional \\
$\calE_{\mathrm{safe}}$ & Running safety/barrier functional \\
$\Psi$ & Barrier-type penalty used in $\calE_{\mathrm{safe}}$ \\
$J(T,u)$ & Free-terminal-time cost functional \\
$\lambda_1,\dots,\lambda_4$ & Cost weights in $J(T,u)$ \\
$\tau_{\calT}$ & First hitting time of the target tube \\
$\mu\in\calP(C([-\tau_{\max},T];\R^{2d}))$ & Path-space law in the continuum
outlook \\
\bottomrule
\end{tabular}
\end{table}

\bibliography{new}    

\end{document}